\documentclass[11pt]{article}
 \usepackage[margin=1in]{geometry}
\usepackage{float}
\usepackage{amsmath,amsfonts,mathrsfs,wasysym,times,dsfont}
\usepackage{amsmath,amssymb,amsfonts,amsthm,amscd}
\usepackage{mathrsfs,latexsym,url,graphicx,enumerate}
\usepackage{booktabs}
\usepackage{array,multirow}
\usepackage{xcolor}
\usepackage{bbm}
\usepackage[utf8]{inputenc}
\usepackage{amsthm,amscd,amsfonts,newlfont}
\usepackage{amssymb, upref, color}
\usepackage{epsf, subfigure, verbatim}
\usepackage{latexsym}
\usepackage[colorlinks]{hyperref}
\usepackage{mathrsfs}
\usepackage{multirow}
\usepackage{lipsum,multicol}
\usepackage{setspace}
\usepackage{xcolor}

\usepackage{enumitem}

\numberwithin{equation}{section}

\theoremstyle{plain}
\newtheorem{exam}{Example}[section]
\newtheorem{theorem}[exam]{Theorem}
\newtheorem{lemma}[exam]{Lemma}
\newtheorem{remark}[exam]{Remark}

\newtheorem{definition}[exam]{Definition}

\begin{document}

\title{Complex dynamics and pattern formation in a diffusive epidemic model with an infection-dependent recovery rate}

\author{Wael El Khateeb\footnotemark[2],\
\ Chanaka Kottegoda\footnotemark[3],\ \ Chunhua Shan\footnotemark[2]\ $^,$\footnotemark[1]\\}
\date{}

\maketitle

\begin{abstract}

A diffusive epidemic model with an infection-dependent recovery rate is formulated in this paper. Multiple constant steady states and spatially homogeneous periodic solutions are first proven by bifurcation analysis of the reaction kinetics. It is shown that the model exhibits diffusion-driven instability, where the infected population acts as an activator and the susceptible population functions as an inhibitor. The faster movement of the susceptible class will induce the spatial and spatiotemporal patterns, which are characterized by $k$-mode Turing instability and $(k_1, k_2)$-mode Turing-Hopf bifurcation. The transient dynamics from a purely temporal oscillatory regime to a spatial periodic pattern are discovered.  The model reveals key transmission dynamics, including asynchronous disease recurrence, spatially patterned waves, and the formation of localized hotspots. The study suggests that spatially targeted strategies are necessary to contain disease waves that vary regionally and cyclically.



\end{abstract}


\footnotetext[2]{Department of Mathematics and Statistics, The University of Toledo, Toledo, OH 43606, USA.}
\footnotetext[3]{Department of Mathematics and Physics,
Marshall University, Huntington, WV 25755, USA.  The research of C. Kottegoda was partially supported by the startup fund from Marshall University. }
\footnotetext[1]{Corresponding author. {\it E-mail address:} chunhua.shan@utoledo.edu (C. Shan).}

\section{Introduction}\label{introduction}

Turing’s pioneering work showing that reaction-diffusion mechanisms can drive biological morphogenesis has stimulated extensive research on spatial pattern formation across biology, physics, chemistry, and ecology \cite{turing1952morphogenesis}.
The first known epidemiological reaction-diffusion model appeared in 1974, when Noble simulated the geographic spread of the bubonic plague across Europe \cite{Noble1974}. Since then, the field has expanded significantly. Later on, Murray and Arcuri introduced biologically realistic models that incorporated spatial and temporal pattern formation in infected disease dynamics \cite{ ArcuriMurray1986, Murray2002}.

Let $N(x, t)=S(x, t)+I(x, t)+R(x, t)$ be the total
population, where $S(x, t)$, $I(x, t)$, and $R(x, t)$ are the numbers of susceptible, infected, and recovered
individuals at position $x\in\Omega$ and at time $t$, respectively, where $\Omega$ is some spatial region. Depending on the manner of diffusion and dispersion, the disease may be restricted to a small region, or it can result in a geographic wave of infections. 
To study how epidemic waves and clusters vary in space and time, we consider a model of a reaction-diffusion system
    \begin{equation}\label{RD_model_original}
    \begin{cases}
        S_t-r_1\Delta S=A-dS-\beta SI, \hskip 3.65cm x\in \Omega, t>0;\\
        I_t- r_2 \Delta I=\beta SI-dI-h(I) I,\hskip 3.2cm x\in \Omega, t>0;\\
        R_t-r_3\Delta R=h(I)I-dR,\hskip 4.05cm x\in \Omega, t>0;\\
        S(x, 0)=S_0(x), I(x, 0)=I_0(x), R(x, 0)=R_0(x),\hskip 0.3cm x\in\Omega;\\
\frac{\partial S}{\partial \nu}=0, \frac{\partial I}{\partial \nu}=0, \frac{\partial R}{\partial \nu}=0, \hskip 4.50cm x\in\partial \Omega, t>0,
    \end{cases}
\end{equation}
where $r_1$, $r_2$ and $r_3$ are the positive diffusion coefficients, $\Omega\subset\mathbb{R}^n$ is open and bounded, and $\nu$ is the unit outward normal vector on the smooth boundary $\partial \Omega$. $\Delta$ is the Laplacian with respect to the spatial variable $x$. The Neumann boundary condition with zero flux simulates the case where a local region implements full isolation or quarantine, with no individuals crossing in or out of a bounded geographic region. $S_0(x)$, $I_0(x)$, and $R_0(x)$ are the initial distributions of the susceptible, infected, and recovered populations.

For epidemiological parameters, $A$ is the recruitment rate of the susceptible population, $d$ is the natural death rate per capita, and $\beta$ is the transmission rate of contacts. The function $h(I)$ is the infection-dependent recovery rate, which takes the form \begin{equation*}\label{mu_formula}
h(I)=\mu_{0}+(\mu_{1}-\mu_{0})b/(I+b),
\end{equation*} where $b$ characterizes the availability of health resources. Parameters $\mu_0$ and $\mu_1$ denote the minimum and maximum per capita recovery rates, respectively. The reader is referred to the work \cite{SZ} for a detailed description of $h(I)$.

Since recovered individuals are assumed to acquire permanent immunity in model \eqref{RD_model_original}, they do not return to the susceptible class. Therefore, the third equation is decoupled from the first two equations of the model \eqref{RD_model_original}. Hence, it suffices to study the following diffusive epidemic model.
\begin{equation}\label{RD model}
    \begin{cases}
        S_t-r_1\Delta S=A-dS-\beta SI, \hskip 2.3cm x\in \Omega, t>0;\\
        I_t- r_2 \Delta I=\beta SI-dI-h(I) I,\hskip 1.8cm x\in \Omega, t>0;\\
        S(x, 0)=S_0(x)\geq 0, I(x, 0)=I_0(x)\geq 0, \hskip 0.3cm x\in\Omega;\\
\frac{\partial S}{\partial \nu}=0, \frac{\partial I}{\partial \nu}=0, \hskip 4.45cm x\in\partial \Omega, t>0.
    \end{cases}
\end{equation}

For reaction-diffusion epidemic models, if the population is assumed to be ideally spatially homogeneous, the corresponding ordinary differential equation (ODE) models are obtained. Since the compartmental models proposed by Kermack and McKendrick in 1927 \cite{Kermack1}, epidemiological ODE models have been extensively formulated and analyzed \cite{AM1991, H1994}. The bilinear, standard, and nonlinear incidence rates are often incorporated into the epidemic models in the last several decades \cite{Alexander2, CS,DD93, HD97,R2011, LLI86,LH87,Huang2019,RW03, THRZ08}, and references therein. Various forms of nonlinear recovery rates are also proposed in epidemic models  \cite{Cui2008, Hu2012, SZ, WR2004} and references therein. Both the nonlinear incidence rates and recovery rates may produce rich dynamics, including multiple equilibria and periodic solutions. 

The study of epidemic diffusive models and dynamics has been conducted by many researchers from different aspects, such as traveling waves, basic reproduction number, free boundary problems, pattern formations, patch models, etc. (See Wang and Wu \cite{WangWu2010}, Liang and Zhao \cite{Liang2007}, Wu and Zou \cite{WuZou2001}, Wang and Wang \cite{WangWang2016}, Chen et al. \cite{Chen2020}, Cui, Lam and Lou \cite{Cui2017}, Wang and Zhao \cite{WangZhao2012}, Lin and Zhu \cite{Lin2017}, Du et al. \cite{Du2024}, Chang et al. \cite{Chang2022}, Sun \cite{Sun2012}, Yi, et al. \cite{YWS2009}, Gao and Ruan \cite{Gao2012}, Xiang et al. \cite{Xiang2024}, Ge et al. \cite{Ge}, Cantrell and Cosner \cite{Cantrell2003}, etc.) and references therein.

If diffusion coefficients of variables differ, specific wavelengths are amplified rather than smoothed out, and spatial heterogeneity follows up \cite{AnJiang2019, ArcuriMurray1986, CM1997, JiangWangCao2019,Murray2002, RovinskyMenzinger1992, SanchezGarduno2019, turing1952morphogenesis, WangWangQi2024,YangSong2016}. Complex patterns can bifurcate out of various bifurcations such as Turing bifurcation, Turing-Hopf bifurcation, double-zero bifurcation, and Turing-Turing bifurcation \cite{AnJiang2019, CJ2020, CJ2022, jiang2020formulation, JiangWangCao2019, LWD, RMP, RovinskyMenzinger1992, SanchezGarduno2019, SR2015,  SJLY, turing1952morphogenesis, WangWangQi2024,YangSong2016, ZW2022}.  Following the framework by Jiang et al. \cite{CJ2022, jiang2020formulation}, we study the spatial, temporal, and spatiotemporal patterns through the $k$-mode Turing bifurcation, $k$-mode Hopf bifurcation, and $(k_1, k_2)$-mode Turing-Hopf bifurcation. We show the existence of multiple constant steady states and spatially homogeneous periodic solutions by bifurcation analysis of the reaction kinetics. For the diffusion-driven instability, we prove that the faster movement of the susceptible class will induce the spatial and spatiotemporal patterns, and the infected population functions as an activator, while the susceptible population serves as an inhibitor.

The paper is organized as follows. In section 2, we analyze the reaction kinetics of model \eqref{RD model} through its constant steady states and spatially periodic solutions, i.e., the equilibrium points and limit cycles of the underlying ODE model.
In Section 3, the spatial and spatiotemporal patterns are studied through the $k$-mode Turing instability, Hopf bifurcation, and $(k_1, k_2)$-mode Turing-Hopf bifurcation. Interesting patterns are found and illustrated. We discuss the transient dynamics between temporal oscillations and spatially periodic patterns and provide their biological interpretation in the last section.


\section{Dynamics of underlying ODE model}

For the reaction-diffusion model \eqref{RD model}, its reaction kinetics (i.e., spatially homogeneous kinetics) are governed by the underlying ODE system
\begin{equation}\label{submodel}
\left\{\begin{aligned}
\frac{dS}{dt}&=A-dS-\beta SI,\\
\frac{dI}{dt}&=\beta SI-dI-h(I)I.
\end{aligned}\right.
\end{equation}
In the following, we will analyze the reaction kinetics of model \eqref{RD model} through its constant steady states and spatially periodic solutions, i.e., the equilibrium points and limit cycles of system \eqref{submodel}.

By Comparison Theorem, one can show that the set $\mathcal{U}=\{(S, I)\in \mathbb{R}^2|S\geq 0, I\geq 0, S+I\leq\frac{A}{d}\}$ is positively invariant, so we restrict system \eqref{submodel} to region $\mathcal{U}$.


\subsection{Constant steady states}

System \eqref{submodel} has a disease-free equilibrium $E_0(\frac{A}{d}, 0)$. By the method of next-generation matrix \cite{DW02}, the basic reproduction number is
$$\mathbb{R}_0=\frac{\beta S}{d+h'(I)I+h(I)}_{\big|(S, I)=(\frac{A}{d}, 0)}=\frac{\beta A}{d(d+\mu_1)}.$$

\begin{theorem} Consider the local and global stability of $E_0$.
\begin{itemize}
    \item[(1)] $E_0$ is locally asymptotically stable if $\mathbb{R}_0<1$ and unstable if $\mathbb{R}_0>1$.

    \item[(2)]  $E_{0}$ is globally
asymptotically stable if $\beta A\leq d(d+\mu_{0})$.
\end{itemize}

\end{theorem}
\begin{proof}
(1) The Jacobian of system \eqref{submodel} at $E_0$ has two eigenvalues $\lambda_1=-d$ and $\lambda_2=(d+\mu_1)(\mathbb{R}_0-1)$. If $\mathbb{R}_0<1$, then $\lambda_2<0$, and if $\mathbb{R}_0>1$, then $\lambda_2>0$. By the stable manifold theorem and Hartman-Grobman theorem, $E_0$ is locally asymptotically stable if $\mathbb{R}_0<1$ and unstable if $\mathbb{R}_0>1$.

(2) Consider the Lyapunov function $I$. The total derivative of $I$ with respect to system \eqref{submodel} is
$$dI/dt=(\beta S-d-h(I))I\leq (\beta A/d-d-\mu_0)I\leq 0.$$
By Lyapunov-Lasalle invariance principle, any orbits in $\mathcal{U}$ approach the largest positively invariant subset of $\{(S, I)\in \mathcal{U}|dI/dt=0\}$, which is the
$S$-axis, where we have $S(t)\rightarrow \frac{A}{d}$ as
$t\rightarrow \infty$. Thus, all orbits in $\mathcal{U}$ converge to
$E_{0}$. Hence, $E_0$ is attractive. Note that $E_0$ is locally asymptotically stable if $\beta A\leq d(d+\mu_{0})$, so $E_{0}$ is globally asymptotically stable.
\end{proof}

For any endemic equilibrium $E(S, I)$, we have $S=\frac{h(I)+d}{\beta}$ and $I$ is the positive root of
\begin{equation}\label{f}
f(I):=A_2I^2+A_1I+A_0=0,
\end{equation}
where
$A_2=(d+\mu_0)\beta$,  $A_1=d^2+(b\beta+\mu_0)d+(b\mu_1-A)\beta$, and $A_0=bd(d+\mu_1)(1-\mathbb{R}_0).$
Equation \eqref{f} has two possible roots
$I_1=\frac{-A_1-\sqrt{\Delta}}{2A_2}$ and $ I_2=\frac{-A_1+\sqrt{\Delta}}{2A_2},$ where $\Delta$ is the discriminant, and
\begin{eqnarray*}
&& \Delta=(d+\mu_1)^2\beta^2b^2+2A(d+2\mu_0-\mu_1)\beta^2 b-2d(d+\mu_1)(d+\mu_0)\beta b\\
&& \hskip 3.6cm  +A^2\beta^2-2d(d+\mu_0)A\beta+d^2(d+\mu_0)^2.
\end{eqnarray*}
Then system \eqref{submodel} has two possible equilibrium points $E_1(S_1, I_1)$ and $E_2(S_2, I_2)$, where $S_i=\frac{h(I_i)+d}{\beta}$, $i=1, 2$. We now establish the existence theorem for endemic equilibrium points.

\begin{theorem}\label{exists}
For system \eqref{submodel}, the following statements hold.
\begin{itemize}\setlength\itemsep{0.cm}
    \item[$(a)$] If $\mathbb{R}_0>1$, system \eqref{submodel} has a unique endemic equilibrium point $E_2$.

\item[$(b)$] If $\mathbb{R}_0=1$ and $A_1<0$, system \eqref{submodel} has a unique endemic equilibrium point $E_2$.

\item[$(c)$] If $\mathbb{R}_0<1$, $\Delta>0$ and $A_1<0$, system \eqref{submodel}  has two endemic equilibria $E_1$ and $E_2$. These two equilibria coalesce into $E^\ast(S^\ast, I^\ast)$ if $\Delta=0$.

\item[$(d)$] System \eqref{submodel} does not have any endemic equilibrium point for other cases.
\end{itemize}
\end{theorem}

We interpret the conditions in Theorem \ref{exists} from a biological point of view. Parameters $b$ and $\beta$ are chosen as bifurcation parameters to explore the impact of health resources and contact transmission rate. The basic reproduction number $\mathbb{R}_{0}=1$ defines a straight line
$C_{0}: \beta=\Phi_0(b)=d(d+\mu_1)/A.$

\begin{figure}[!ht]
\begin{center}
{\includegraphics[angle=0, width=0.35\textwidth]{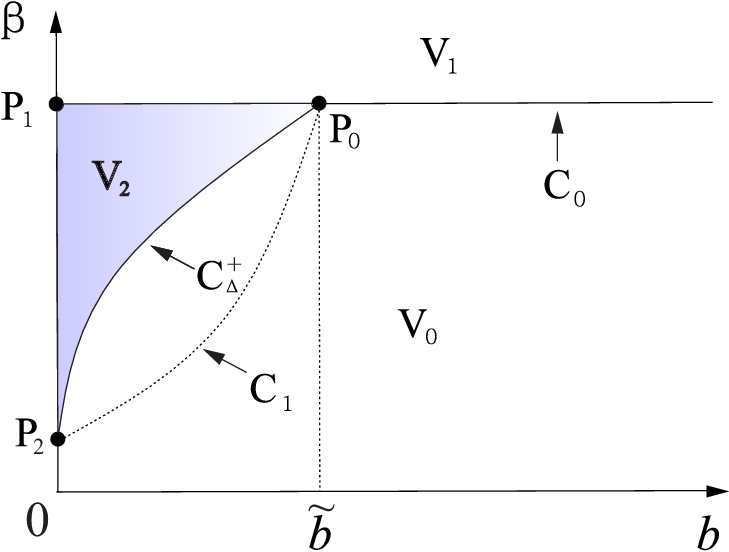}}
 \caption{Bifurcations of endemic equilibrium points in $(b, \beta)$-plane}\label{EEE}
\end{center}\end{figure}

The condition $A_1=0$ defines a hyperbola
\begin{eqnarray*}
C_{1}:
\beta=\Phi_1(b)=\frac{d(d+\mu_0)}{A-b(d+\mu_1)}.
\end{eqnarray*}
The curve $C_0$ intersects $\beta$-axis at $P_1(0, \frac{d(d+\mu_1)}{A})$ and $C_1$ at $P_0(\tilde{b}, \frac{d(d+\mu_1)}{A})$, where $\tilde{b}=\frac{A(\mu_1-\mu_0)}{(d+\mu_1)^2}$. The curve $C_1$ intersects $\beta$-axis at $P_2(0, \frac{d(d+\mu_0)}{A})$. The condition $\Delta=0$ defines a quadratic curve
\begin{eqnarray*}
C_{\Delta}^{\pm}: \beta=\Phi_{\Delta}^{\pm}(b)=\dfrac{d(d+\mu_0)(bd+b\mu_1+A)\pm2(d+\mu_0)d\sqrt{(\mu_1-\mu_0)Ab}}{A^2+2b(2\mu_0+d-\mu_1)A+b^2(d+\mu_1)^2},
\end{eqnarray*}
where $C_{\Delta}^{+}$ intersects $C_0$ at $P_0$ and $\beta$-axis at $P_2$.
It is easy to verify that $\Phi_{\Delta}^{-}(b)\leq \Phi_1(b)\leq \Phi_{\Delta}^{+}(b)$ for $b\geq 0$. By the conditions in Theorem \ref{exists} $(c)$, $\Delta=0$ corresponds to the curve $C_{\Delta}^{+}$. The two curves $C_0$ and $C_{\Delta}^+$ divide the first quadrant of $(b, \beta)$-plane into three subregions
\begin{eqnarray*}
&& V_0=\{(b, \beta)|0<\beta<\Phi_{\Delta}^{+}(b), 0<b<\tilde{b}\}\cup\{(b, \beta)|0<\beta<\Phi_0(b), b\geq\tilde{b}\},\\
&& V_1=\{(b, \beta)|\beta>\Phi_0(b), b>0\},\quad V_2=\{(b, \beta)| \phi_{\Delta}^{+}(b)<\beta<\Phi_0(b), 0<b<\tilde{b}\}.
\end{eqnarray*}
There are $0$, $1$ and $2$ simple endemic equilibria in $V_0$, $V_1$ and $V_2$, respectively. Furthermore,  $E_2$ exists in $V_1$, while $E_1$ and $E_2$ exist in $V_2$. The equilibrium $E_1$ disappears when crossing the line segment $\overline{P_0P_1}$ from $V_2$ to $V_1$,
On the arc $\widehat{P_0P_2}$ of $C_{\Delta}^{+}$, two equilibria $E_1$ and $E_2$ coalesce into an equilibrium $E^{\ast}$. Hence, we have the following theorem.

\begin{theorem}
System \eqref{submodel} undergoes a forward bifurcation if $\beta=\Phi_0(b)$ and $b>\tilde{b}$. System \eqref{submodel} undergoes a backward bifurcation if $\beta=\Phi_0(b)$ and $b<\tilde{b}$. System \eqref{submodel} undergoes a saddle-node bifurcation if $\beta=\Phi_{\Delta}^{+}(b)$ and $0<b<\tilde{b}$.
\end{theorem}

\subsection{Spatially homogeneous periodic solution}

The Jacobian of system \eqref{submodel} at any endermic equilibrium $E(S, I)$ is
\begin{equation*}
 \displaystyle J_0(S, I)=\left[\begin{array}{cc}
   -(d+\beta I) & -(h(I)+d)\\
   \beta I    & -h'(I)I
        \end{array}\right].
\end{equation*}
Then the characteristic equation at $E(S, I)$ is
\begin{eqnarray*}\label{characteristic}
P_0(\lambda)=\lambda^2+B_1\lambda+B_0=0,
\end{eqnarray*}
where $B_1(I)=d+\beta I+h'(I)I$ and $B_0(I)=(d+\beta I)h'(I)I+\beta(h(I)+d)I$.

\begin{theorem}\label{E1 saddle E2 antisaddle}
$E_1$ is a saddle point. $E_2$ is an anti-saddle point.
\end{theorem}
\begin{proof}
Firstly, we prove the assertion:  for any positive equilibrium $E(S, I)$, $\displaystyle B_0(I)=I(I+b)^{-1}f'(I)$.

Let $\tilde{f}(S, I)$ and $\tilde{g}(S, I)$ be the right-hand sides of the first and second equations of system \eqref{submodel}, respectively. Solve
$\tilde{g}(S, I)=0$ for $S$ and we obtain $S(I)=\frac{h(I)+d}{\beta}$.
Substituting it into $\tilde{f}$ yields the function
$F(I)=\tilde{f}(S(I), I)).$ By chain rule and Implicit Function Theorem, we have
\begin{eqnarray*}\label{jacobian}
F'(I)=\frac{\partial \tilde{f}}{\partial S}\frac{d S}{d I}+\frac{\partial \tilde{f}}{\partial I}=-\frac{\partial \tilde{f}}{\partial S}\frac{\frac{\partial \tilde{g}}{\partial I}}{\frac{\partial \tilde{g}}{\partial S}}+\frac{\partial \tilde{f}}{\partial I}=-\frac{det(J_0(I))}{\frac{\partial \tilde{g}}{\partial S}}.
\end{eqnarray*}
A straightforward calculation yields
$
F(I)=-\beta^{-1}(I+b)^{-1}f(I)$ and $F'(I)=-\beta^{-1}(I+b)^{-1}f'(I)$.
Then
$$B_0(I)=det(J_0(I))=-\frac{\partial \tilde{g}}{\partial S} F'(I)=-\beta I (-\beta^{-1})(I+b)^{-1}f'(I)=I(I+b)^{-1}f'(I),$$
and the desired assertion is obtained.

Notice that $f'(I_1)<0$ and $f'(I_2)>0$,  and by the assertion above, we have $B_0(I_1)<0$ and $B_0(I_2)>0$. Hence, $E_1$ is a saddle point, while $E_2$ is not a saddle point.
\end{proof}




The stability of $E_2$ depends on the sign of $B_1(I_2)$.  Now we rewrite $B_1(I)=\Psi(I)(I+b)^{-2}$, where $$\Psi(I)=\beta I^3+(2b\beta+d)I^2+(b^2\beta+2bd-b\mu_1+b\mu_0)I+b^2d.$$ The sign of
$\Psi(I)$ is studied in the Lemma \ref{root} below, and the proof is provided in the Appendix.

\begin{lemma}\label{root}  Let $\omega=[8b^2\beta^2+20b\beta d-d^2+\sqrt{d(8b\beta+d)^3}](8b\beta)^{-1}$, and $\bar{\omega}=\mu_1-\mu_0$. 

\begin{itemize} \setlength\itemsep{0.cm}
    \item[$(1)$] $\Psi(I)$ always has two humps.

\item[$(2)$] If $\bar{\omega}<b\beta+2d$, both humps are on the left-side of $I=0$. Hence, $\Psi(I)>0$ for $I>0$.

\item[$(3)$] If $\bar{\omega}=b\beta+2d$, one humps is on the left-side of $I=0$, and the other hump sits on $I=0$. Hence, $\Psi(I)>0$ for $I>0$.

\item[$(4)$] If $\bar{\omega}>b\beta+2d$, two humps are on the different sides of $I=0$.
\vskip -0.2cm
\begin{itemize}\setlength\itemsep{0.cm}
    \item If $\bar{\omega}<\omega$, then $\Psi(I)>0$ for $I>0$.

    \item If $\bar{\omega}=\omega$, there exists $\tilde{I}>0$ such that $\Psi(\tilde{I})=\Psi'(\tilde{I})=0$, and $\varphi(I)>0$ for $I\in(0, \tilde{I})\cup (\tilde{I}, \infty)$.

    \item If $\bar{\omega}>\omega$, there exist  $\check{I}>0$ and $\hat{I}>0$ such that $\Psi(\check{I})=\Psi(\hat{I})=0$, $\Psi'(\check{I})<0$, $\Psi'(\hat{I})>0$, and $\Psi(I)>0$ for $I\in (0, \check{I})\cup(\hat{I}, \infty)$ and $\Psi(I)<0$ for $I\in(\check{I}, \hat{I})$.
\end{itemize}

\end{itemize}

\end{lemma}

\begin{theorem}\label{Tr(J)<0}
$E_2$ is unstable if  $\bar{\omega}>\omega$ and $I_2\in (\check{I}, \hat{I})$. $E_2$ is locally asymptotically stable if either of the following statements holds.

(a) $\bar{\omega}<\omega$; (b) $\bar{\omega}=\omega$ and $I_2\neq \tilde{I}$; (c) $\bar{\omega}>\omega$ and $I_2\in (0, \check{I})\cup(\hat{I}, \infty)$.
\end{theorem}

\begin{proof}
This is the direct consequence of Lemma \ref{root}.
\end{proof}

\begin{theorem}
Let $\mathbb{R}_0>1$. If $E_2$ is unstable, then system \eqref{submodel} has a stable periodic solution. If $E_2$ is stable, then system \eqref{submodel} either does not have a periodic solution, or has at least one periodic solution.
\end{theorem}

\begin{proof}
If $\mathbb{R}_0>1$,  $E_2$ is the unique equilibrium in the attracting invariant set $U$, and the stable manifold of $E_0$ is the $S$-axis. The desired results follow from the Poincar\'e-Bendixson Theorem.
\end{proof}




In the following, we study the periodic solutions bifurcating from a Hopf bifurcation around $E_2$. Let $\tau=\displaystyle\int_0^t(I(s)+b)^{-1}ds$. Then system \eqref{submodel} is $C^\infty$-equivalent to the cubic polynomial system
\begin{equation*}\label{submodel2}
\left\{\begin{aligned}
\frac{dS}{d\tau}&=(A-dS-\beta SI)(I+b),\\
\frac{dI}{d\tau}&=(\beta S-d)(I+b)I-(\mu_0I+\mu_1b)I.
\end{aligned}\right.
\end{equation*}
Let $x=S-S_2, y=I-I_2$. Then
\begin{equation*}\label{submodel3}
\left[\begin{array}{rcl}
x'(\tau)\\
y'(\tau)
\end{array}\right]=\left[\begin{array}{cc}
\bar{j}_{11} &\bar{j}_{12}\\
\bar{j}_{21} &\bar{j}_{22}
\end{array}\right]\left[\begin{array}{rcl}
x\\
y
\end{array}\right]+\left[\begin{array}{l}
\bar{a}_{11}xy+\bar{a}_{02}y^2+\bar{a}_{12}xy^2\\
\bar{b}_{11}xy+\bar{b}_{02}y^2+\bar{b}_{12}xy^2
\end{array}\right],
\end{equation*}
where
\begin{equation*}
\left.\begin{array}{ll}
\bar{j}_{11}=  -(d+\beta I_2)(I_2+b), &  \bar{j}_{12}=-(h(I_2)+d)(I_2+b),\\
\bar{j}_{21}= \beta I_2(I_2+b), &  \bar{j}_{22}=-h'(I_2)I_2(I_2+b),
\end{array}\right.
\end{equation*}
and
\begin{equation*}
\left.\begin{array}{lll}
\bar{a}_{11}=-(b\beta+2\beta I_2+d), &  \bar{a}_{02}=-(h(I_2)+d), & \bar{a}_{12}=-\beta, \\
\bar{b}_{11}=b\beta+2\beta I_2, &  \bar{b}_{02}=\beta S_1-d-\mu_0, & \bar{b}_{12}=\beta.
\end{array}\right.
\end{equation*}

According to the formula by Andronov et al. \cite{Andronov}, the first Lyapunov number is
\begin{equation}\label{ll1}
\left.\begin{aligned}
L_1= & \dfrac{-3\pi[\bar{j}_{11}\bar{j}_{21}(\bar{a}^2_{11}+\bar{a}_{11}\bar{b}_{02}+\bar{a}_{02}\bar{b}_{11}-2\bar{b}^2_{02})+\bar{j}_{11}\bar{j}_{12}(\bar{b}^2_{11}+\bar{a}_{11}\bar{b}_{02})+\bar{j}_{12}\bar{j}_{21}\bar{b}_{11}\bar{b}_{02}]}{2\bar{j}_{12}(\bar{j}_{11}\bar{j}_{22}-\bar{j}_{12}\bar{j}_{21})^{\frac{3}{2}}}\\
      &-
\dfrac{3\pi[\bar{j}^2_{21}(\bar{a}_{11}\bar{a}_{02}+2\bar{a}_{02}\bar{b}_{02})-2\bar{j}^2_{11}\bar{b}_{11}\bar{b}_{02}-(\bar{j}^2_{11}+\bar{j}_{12}\bar{j}_{12})(2\bar{j}_{11}\bar{b}_{12}+\bar{j}_{21}\bar{a}_{21})]}{2\bar{j}_{12}(\bar{j}_{11}\bar{j}_{22}-\bar{j}_{12}\bar{j}_{21})^{\frac{3}{2}}}.\\
\end{aligned}\right.
\end{equation}



\begin{theorem}\label{Hopf} Let $B_1(I_2)=0$ (i.e., $I_2=\check{I} \text{or}\ \hat{I}$) and $B_1'(I_2)\neq 0$.
System \eqref{submodel} undergoes a supercritical Hopf bifurcation at $E_2$ if $L_1<0$, and undergoes a subcritical Hopf bifurcation at $E_2$ if $L_1>0$.
\end{theorem}

\begin{proof}
Consider $A$ as a bifurcation parameter, and assume that $B_1(I_2)=0$ for $A=A^\ast$. Then $J(I_2)$ has a pair of complex conjugate eigenvalues $\lambda_{1,2}=\lambda_{1,2}(A)$ in a small neighborhood of $A^\ast$, which is continuously differentiable with respect to $A$, and $Re(\lambda_{1,2}(A^\ast))=0$ and $Im(\lambda_{1,2}(A^\ast))\neq 0$. Then
$$\dfrac{d(Re(\lambda_{1,2}(A^\ast)))}{dA}=-\frac{1}{2}B'_1(I_2)\frac{\partial I_2}{\partial A}=-\frac{[\beta(I_2+b)]B'_1(I_2)}{2\sqrt{\Delta}}\neq 0.$$
Hence, system \eqref{submodel} undergoes a Hopf bifurcation at $E_2$ for $B_1(I_2)=0$ and $B'_1(I_2)\neq 0$.
\end{proof}



\begin{figure}[!ht]
\begin{center}
{\includegraphics[angle=0, width=0.75\textwidth]{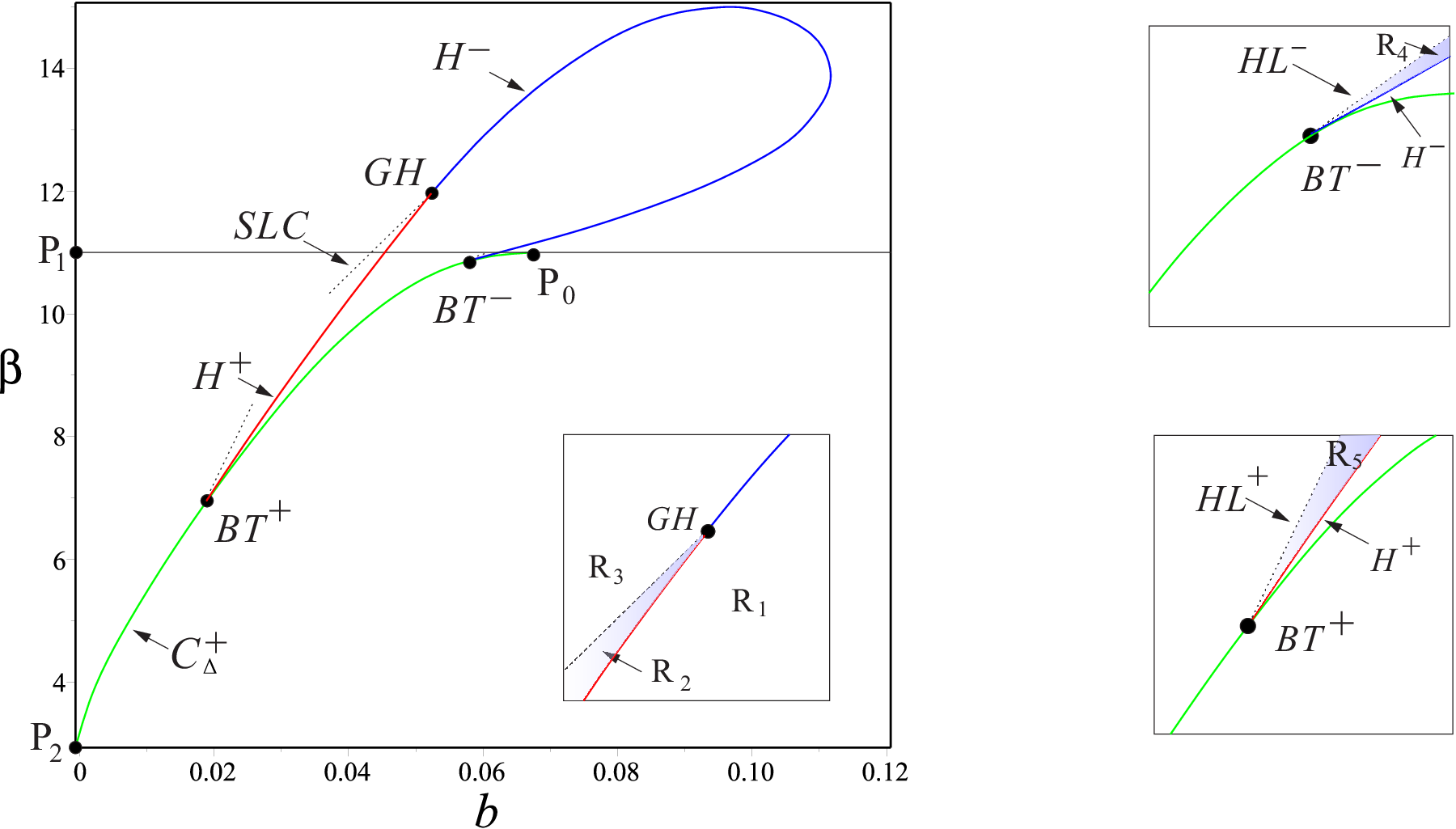}}
 \caption{Bifurcation diagram in $(b, \beta)$-plane. Three small panels are the magnified view in a small neighborhood of $GH$, $BT^-$, and $BT^+$, respectively.} \label{HHH}
\end{center}\end{figure}

\noindent {\bf Example 1}. Let $d=1,\mu_1=10, \mu_0=2, A=1$. We sketch the Hopf bifurcation diagram in the $(b, \beta)$-plane by numerical simulation.

The Hopf bifurcation curve is denoted by
$H=\{(b, \beta)|B_1(I_2)=0\}.$
The curve $H$ intersects the saddle-node bifurcation curve $C_\Delta^+$ at two points $BT^-$ and $BT^+$. Here, $BT^-$ and $BT^+$ denote the supercritical and subcritical Bogdanov-Takens bifurcation points, respectively. In a {\it small neighborhood} of point $BT^-$ there exists a region $R_4$ between the Hopf bifurcation curve $H^-$ and the homoclinic bifurcation curve $HL^-$, and the system \eqref{submodel} has a unique limit cycle in $R_4$, which is stable (see the horizontal panel of Fig. \ref{HHH}). In a {\it small neighborhood} of point $BT^+$, there exists a region $R_5$  between the Hopf bifurcation curve $H^+$ and the homoclinic bifurcation curve $HL^+$, and system \eqref{submodel} has a unique limit cycle in $R_4$, which is unstable (see lower right panel of Fig. \ref{HHH}).

We trace the first Lyapunov number $L_1$ by formula \eqref{ll1} along the curve $H$ and find that there is a point $$GH: (b, \beta)\approx (0.052935, 12.084927)$$ at which $L_1=0$, 
and a generalized Hopf bifurcation of codimension 2 occurs. Hence, $H=H^-\cup\{GH\}\cup H^+$, where $H^-$ is the arc $\widehat{GHBT^-}$ (blue curve), and $H^+$ is the arc $\widehat{GHBT^+}$ (red curve).

\begin{figure}[H]
\begin{center}
\subfigure[$E_2$ is stable and no periodic solution in $R_3$]
{\includegraphics[angle=-90,width=0.47\textwidth]{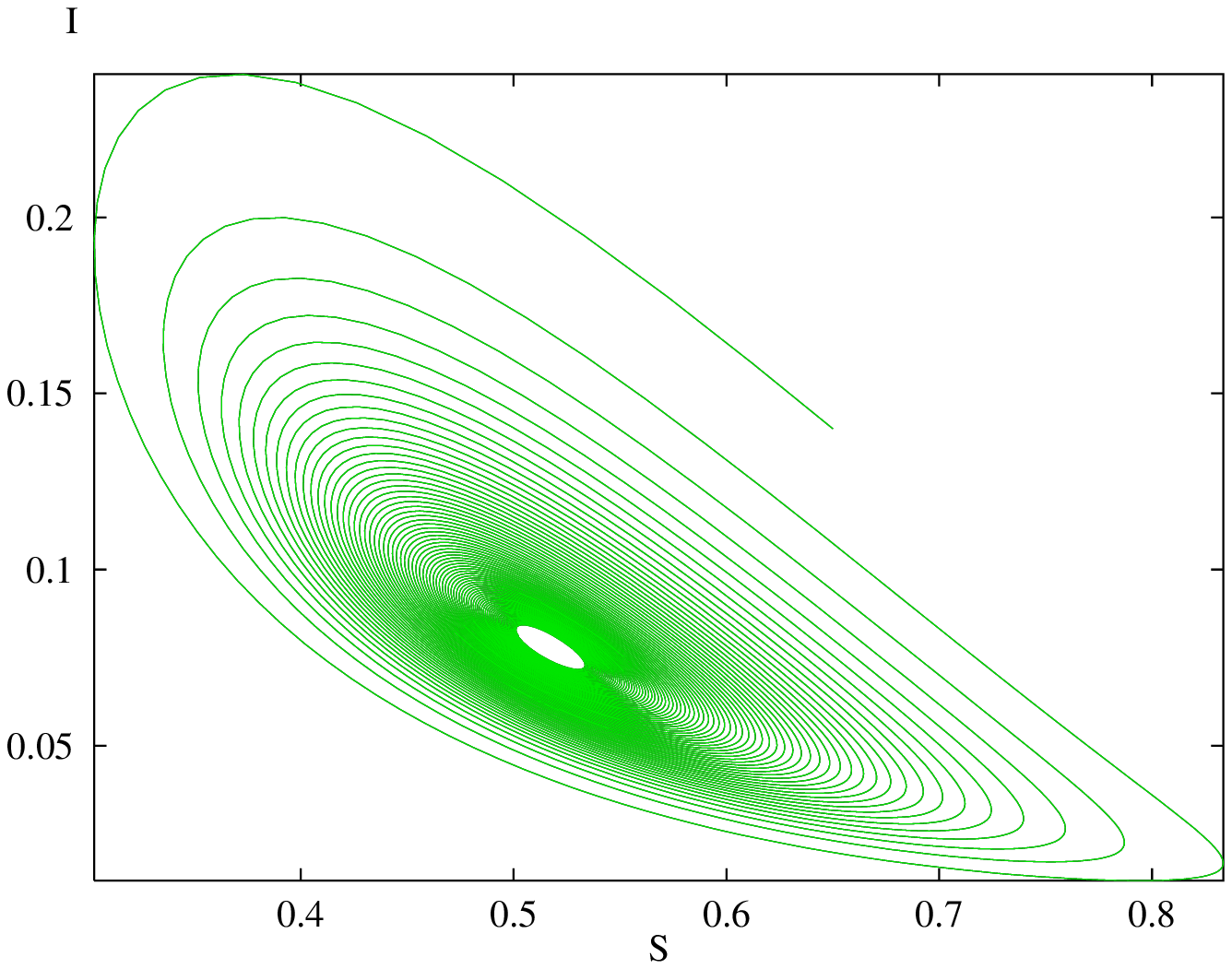}} \ \ \ \
\subfigure[One periodic solution in $R_1$]
{\includegraphics[angle=-90,width=0.47\textwidth]{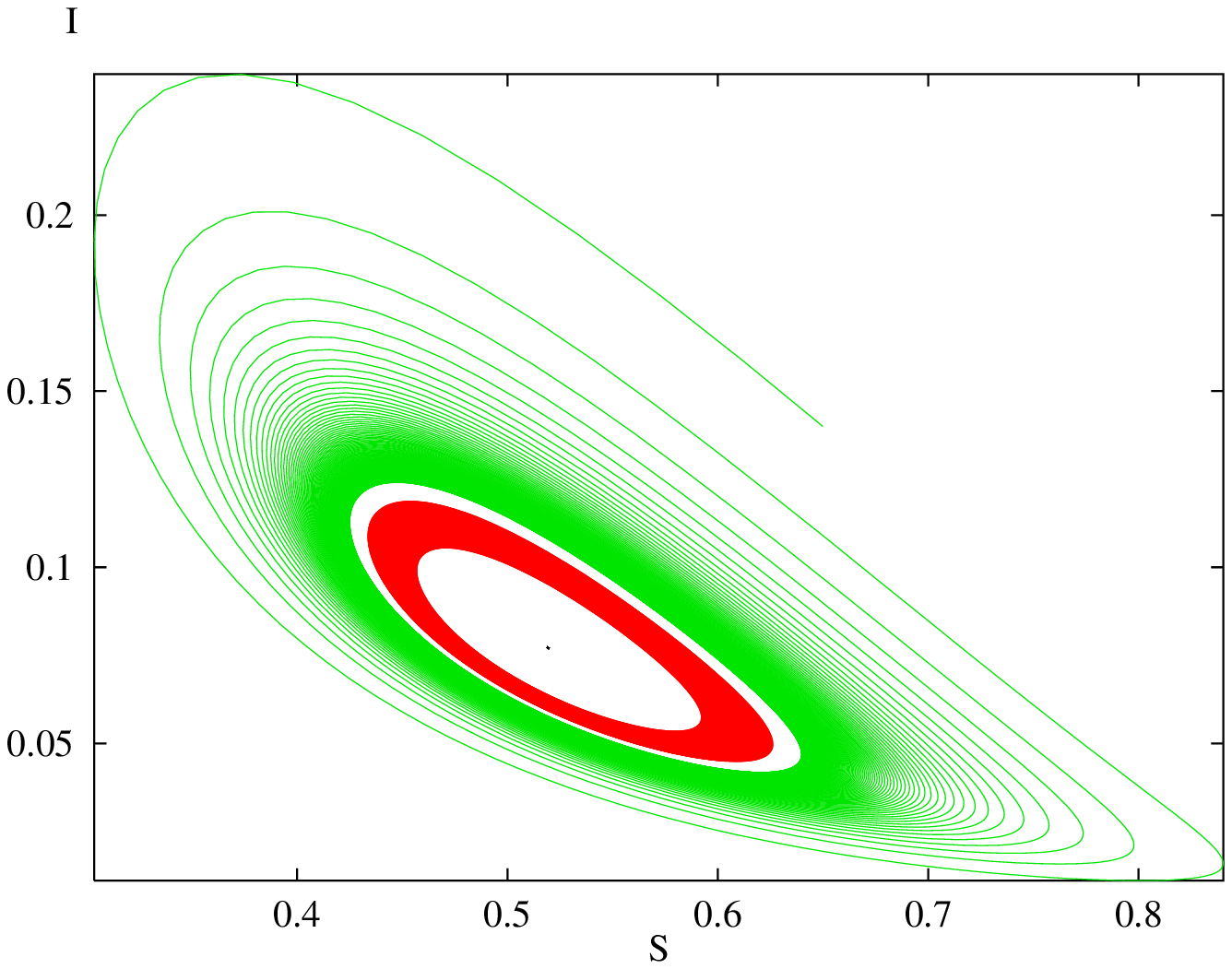}} \ \ \ \
\subfigure[Two periodic solutions in $R_2$]
{\includegraphics[angle=-90,width=0.47\textwidth]{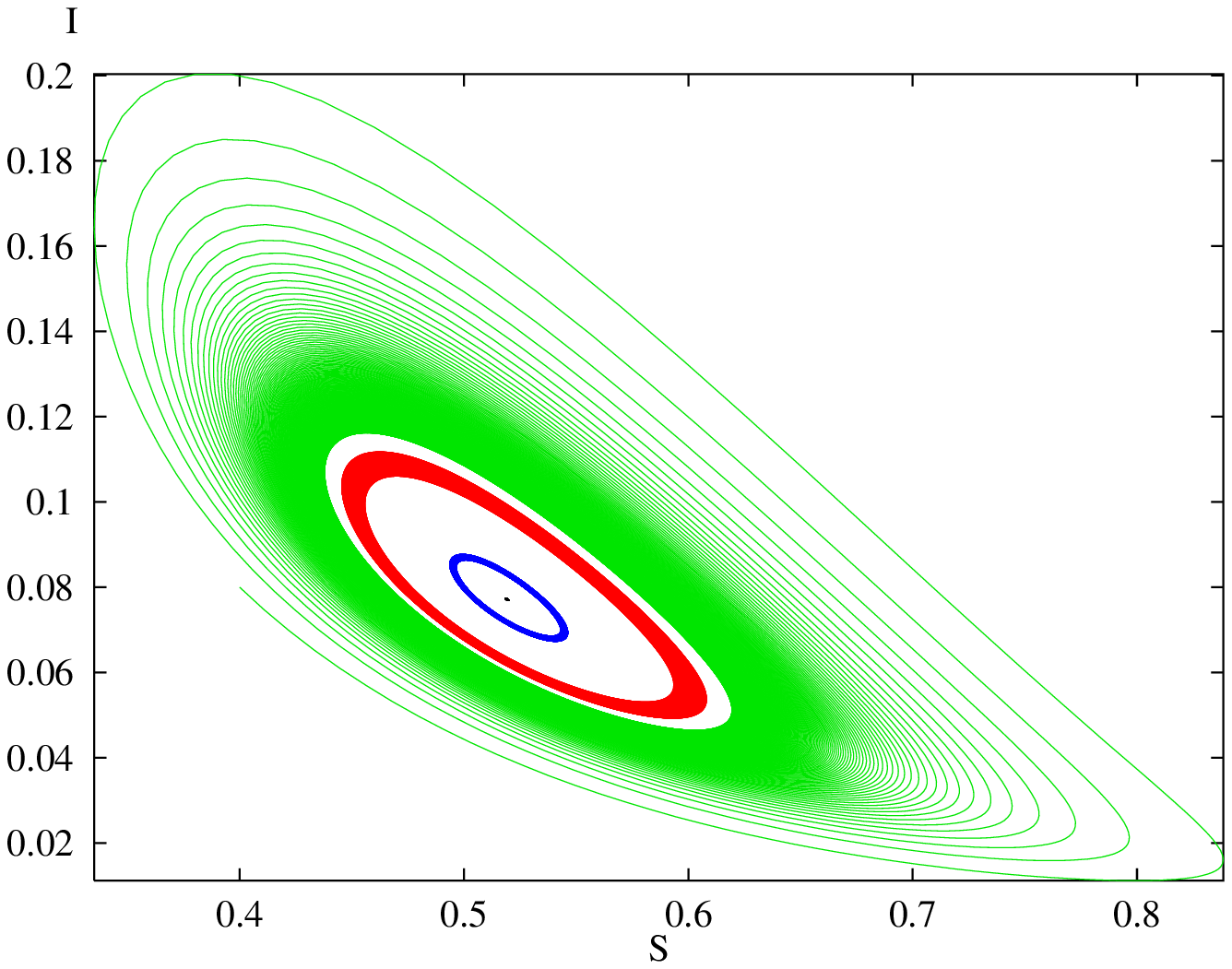}}
\caption{{Phase portraits near the generalized Hopf bifurcation point $GH$. (a) $(b, \beta)\in R_3$. (b) $(b, \beta)\in R_1$. (c) $(b, \beta)\in R_2$.}}\label{HH}
\end{center}
\end{figure}

The Hopf bifurcation curve $H$ and the curve of saddle-node bifurcation of limit cycles $SLC$ divide the {\it small neighborhood} of point $GH$ into three generic regions $R_1$, $R_2$, and $R_3$ (see left panel of Fig. \ref{HHH}). There is no limit cycle in region $R_3$, and $E_2$ is locally asymptotically stable. System \eqref{submodel} undergoes a supercritical Hopf bifurcation when crossing the curve $H^-$, and $E_2$ becomes unstable, and there is a stable limit cycle in region $R_1$. The system \eqref{submodel} undergoes a subcritical Hopf bifurcation when crossing the curve $H^+$, and $E_2$ becomes stable, and there is an unstable limit cycle (which is contained inside the stable limit cycle) in region $R_2$. So there are two limit cycles in $R_2$. These two limit cycles coalesce on the curve $SLC$ and disappear by the saddle-node bifurcation of limit cycles when crossing the curve $SLC$ into region $R_3$.

We sketch the phase portraits in regions $R_1$, $R_2$, and $R_3$ in Fig. \ref{HH}.

$(a)$ $(b, \beta)\in R_3$, where $\beta=12, b=0.0522$. In this case, $E_2$ is a stable equilibrium, and there is no limit cycle. The orbit (green curve) through the point $(0.65, 0.14)$  spirals inward and approaches $E_2$.

$(b)$ $(b, \beta)\in R_1$, where $\beta=12, b=0.052277264$. In this case, $E_2$ is an unstable equilibrium, and there is one stable limit cycle. The orbit (green curve) through the point $(0.65, 0.14)$  spirals inward and approaches the stable limit cycle. The orbit (red curve) through the point $(0.4, 0.08)$  spirals outward and approaches the stable limit cycle.

$(c)$ $(b, \beta)\in R_2$, where $\beta=12, b=0.052264417$. In this case, $E_2$ is a stable equilibrium, and there are two limit cycles, the inner one being unstable and the outer one being stable. The orbit (green curve) through the point $(0.4, 0.08)$  spirals inward and approaches the stable limit cycle. The orbit (red curve) through the point $(0.59, 0.054)$  spirals outward and approaches the stable limit cycle.  The orbit (blue curve) through the point $(0.51, 0.086)$  spirals inward and approaches $E_2$.




\begin{remark}
The limit cycles of the underlying ODE system \eqref{submodel} are spatially homogeneous periodic solutions of the reaction-diffusion model \eqref{RD model}. The Hopf bifurcation of the underlying ODE system \eqref{submodel} is the $0$-mode Hopf bifurcation of the reaction-diffusion model \eqref{RD model}. See section 3.3 for the definition of a $k$-mode Hopf bifurcation.
\end{remark}

\section{Pattern formation of reaction-diffusion model}

As a preliminary study, we consider the reaction-diffusion system \eqref{RD model} with the spatial variable $x\in\Omega$, where $\Omega=[0,\ell\pi]$, and $\ell\in\mathbb{R}^+$. The Neumann boundary condition with zero flux becomes
$$S_x(0,t)=S_x(\ell\pi,t)=I_x(0,t)=I_x(\ell\pi,t)=0.$$
First, we investigate the linearization of the reaction-diffusion system at the constant steady states as follows. 

\subsection{Preliminaries}

Define the real-valued Sobolev space
$$X=\{(u, v)\in H^2(0,\ell\pi)\times H^2(0,\ell\pi)|u_x=v_x=0 \text{ at } x=0,\ell\pi\},$$ where $H^2(0,\ell \pi)$ is the space of functions on $(0,\ell \pi)$
 that have square-integrable derivatives up to second order. We further define $X_{\mathbb{C}}=X\oplus iX$ to be the complexification of $X$, and introduce the linear operator $\mathcal{L}$ on $X_{\mathbb{C}}$ as
 $$\mathcal{L}\begin{bmatrix}
     u\\
     v
 \end{bmatrix}=\begin{bmatrix}
     r_1&0\\
     0&r_2
 \end{bmatrix}\begin{bmatrix}
     u_{xx}\\
     v_{xx}
 \end{bmatrix}+J_0(\bar{S}, \bar{I})\begin{bmatrix}
     u\\
     v
 \end{bmatrix},$$ where $(\bar{S}, \bar{I})$ is an equilibrium point of the underlying ODE system \eqref{submodel}, and $J_0(\bar{S}, \bar{I})$ is the Jacobian of system \eqref{submodel} at $(\bar{S}, \bar{I})$.  Then the linearization of the reaction-diffusion system \eqref{RD model} at  $(\bar{S}, \bar{I})$ is 
 \begin{equation}\label{linear_PDE}
   \begin{bmatrix}
     u_t\\
     v_t
 \end{bmatrix} =\mathcal{L}\begin{bmatrix}
     u\\
     v
 \end{bmatrix}.
 \end{equation}
Let $\mathbb{Z}^+$ be the set of all positive integers. Given the homogeneous Neumann boundary condition, the Fourier series solution of the linear system \eqref{linear_PDE} takes the form
 \begin{equation} \label{eigenfunctions}
\begin{bmatrix}
     u(x, t)\\
     v(x, t)
 \end{bmatrix}=\sum_{k=0}^\infty e^{\lambda t}\cos\frac{k}{\ell}x\begin{bmatrix}
     \phi_k\\
     \psi_k
 \end{bmatrix}, \end{equation}
 where $k\in\mathbb{Z}^+\cup\{0\}$, while $\frac{k}{\ell}$ is known as the wavenumber which
 indicates how rapidly the oscillations occur in space, with higher values corresponding to shorter wavelengths (i.e., finer spatial structures). Refer to \cite{Murray2002} for more details. $\phi_k$ and $\psi_k$ are real constants, and $\lambda=\lambda(k)$ is an eigenvalue of the matrix
  \begin{equation}\label{reduced equation}
     J_k(\bar{S}, \bar{I})=J_0(\bar{S}, \bar{I})-\frac{k^2}{\ell^2}\begin{bmatrix}
     r_1&0\\
     0&r_2
 \end{bmatrix},\ \ k\in\mathbb{Z}^+\cup\{0\}.
 \end{equation}
The characteristic equation of the matrix $J_k(\bar{S}, \bar{I})$ is given by \begin{equation}\label{Characteristic equation for RD}
   P_k(\lambda)=  \lambda^2-\mathcal{T}_k\, \lambda+ \mathcal{D}_k=0,
 \end{equation}
 where $\mathcal{T}_k$ and $\mathcal{D}_k$ are the trace and determinant of $J_k(\bar{S}, \bar{I})$ respectively. It should be noted that $J_k(\bar{S}, \bar{I})=J_0(\bar{S}, \bar{I})$ for $k=0$.
 \begin{theorem}\label{T31}
The constant steady state $E_0(\frac{A}{d}, 0)$
of system \eqref{RD model} is stable if $\mathbb{R}_0<1$ and unstable if $\mathbb{R}_0>1$.
 \end{theorem}
 \begin{proof}
Evaluating the matrix $J_k$ at $E_0$ yields
  $$J_k(\frac{A}{d}, 0)=\left[\begin{array}{cc}
-(d+\frac{r_1k^2}{\ell^2}) & -\left(d +\mu_1 \right)\mathbb{R}_0
\\
 0 & (d+\mu_1)(\mathbb{R}_0-1)-\frac{r_2k^2}{\ell^2}
\end{array}\right].$$
There are two eigenvalues $\lambda_1=-(d+\frac{r_1k^2}{\ell^2})<0$ and $\lambda_2=(d+\mu_1)(\mathbb{R}_0-1)-\frac{r_2k^2}{\ell^2}$. If $\mathbb{R}_0<1$, $\lambda_2<0$ fro all $k\in\mathbb{Z}^+\cup\{0\}$, then $E_0$ is stable. If $\mathbb{R}_0>1$, then $\lambda_2>0$ for $k=0$, and $E_0$ is unstable.
\end{proof}
\begin{remark}\label{R32}
By Theorem \ref{E1 saddle E2 antisaddle}, the constant steady state $E_1$ of system \eqref{RD model} is always unstable regardless of the diffusion effect, since $J_0(S_1, I_1)$ always has a positive eigenvalue. We focus on the constant steady state $E_2$, and investigate possible stability changes driven by the diffusion.
\end{remark}


\subsection{Diffusion-driven instability}

In the following, we explore the spatially inhomogeneous pattern through the Turing instability.

\begin{definition} \cite{jiang2020formulation}
Let $k\in \mathbb{Z}^+$ be a positive integer and fixed. Let $\xi\in\mathbb{R}$ be a bifurcation parameter. A $k$-mode Turing bifurcation occurs if the Jacobian $J_k$ has a unique real eigenvalue $\lambda=\lambda(k)$ crossing zero transversally provided that the spectral gap condition holds, that is, all remaining eigenvalues of $J_n (n=0, 1,\cdots)$ are to the left of the imaginary axis as $\xi$ passing through $\xi=\xi^\ast\in \mathbb{R}$.
\end{definition}

The matrix $J_k (k\in\mathbb{Z}^+)$, evaluated at $E_2$ is $$J_k=\begin{bmatrix}
    \delta_1-\frac{r_1 k^2}{\ell^2} & \delta_2\\
    \delta_3 &\delta_4-\frac{r_2 k^2}{\ell^2}
\end{bmatrix},$$
where $\delta_1=-(\beta I_2+d)<0$, $\delta_2=-(h(I_2)+d)<0$, $\delta_3=\beta I_2>0$ and $\delta_4=-h'(I_2)I_2>0$ are the entries of the Jacobian $J_0(S_2,I_2)$. Assume either of conditions (a), (b), and (c) of Theorem \ref{Tr(J)<0} holds, i.e., $E_2$ is stable for the underlying ODE model \eqref{submodel}. Then
$$\mathcal{T}_k=\delta_1+\delta_4-\frac{(r_1+r_2)k^2}{\ell^2}, \quad \mathcal{D}_k=\frac{r_1r_2}{\ell^4}k^4-\frac{(\delta_1 r_2 +\delta_4 r_1) }{\ell^2}k^2+\delta_1\delta_4-\delta_2\delta_3,$$
and $\mathcal{T}_0=\delta_1+\delta_4<0$ and $\mathcal{D}_0=\delta_1\delta_4-\delta_2\delta_3>0$.

Define $\gamma=\frac{r_2}{r_1}$, and set the discriminant of $\mathcal{D}_k$ with respect to $k^2$ equal to zero, then we obtain
$$\delta_1^2\gamma^2+2(2\delta_2\delta_3-\delta_1\delta_4)\gamma+ \delta_4^2=0,$$
which has two positive roots
$$\gamma_{\pm}=\frac{\delta_1\delta_4-2\delta_3\delta_2\pm 2\sqrt{-\delta_3\delta_2(\delta_1\delta_4-\delta_3\delta_2)}}{\delta_1^2},$$ where $0<\gamma_-<\gamma_+$. Let $\bar{\gamma}=-\frac{\delta_4}{\delta_1},$ and then
$$\gamma_--\bar{\gamma}=\frac{2\delta_1\delta_4-2\delta_3\delta_2-2\sqrt{-\delta_3\delta_2(\delta_1\delta_4-\delta_3\delta_2)}}{\delta_1^2}<0,$$ and
$$\gamma_+-\bar{\gamma}=\frac{2\delta_1\delta_4-2\delta_3\delta_2+2\sqrt{-\delta_3\delta_2(\delta_1\delta_4-\delta_3\delta_2)}}{\delta_1^2}>0.$$
Hence, $0<\gamma_-<\bar{\gamma}<\gamma_+$.

If $\gamma\geq \bar{\gamma}$, then $\delta_1r_2+\delta_4r_1\leq 0$, which yields $\mathcal{D}_k>0$ for all $k\in\mathbb{Z}^+$.

If $\gamma_-<\gamma<\bar{\gamma}$, then the discriminant of $\mathcal{D}_k$ is negative and therefore $\mathcal{D}_k>0$ for all $k\in\mathbb{Z}^+$.

If $\gamma=\gamma_-$, then the discriminant of $\mathcal{D}_k$ is zero and therefore, $\mathcal{D}_k\geq 0$ for all $k\in\mathbb{Z}^+$

If $\gamma<\gamma_-$, then the discriminant of $\mathcal{D}_k$ is positive. Hence, $\mathcal{D}_k$ may change sign, and the steady state $E_2$ loses its stability. To study the instability driven by the diffusion on the steady state $E_2$, we rewrite $\mathcal{D}_k$ as
$$\mathcal{D}_k=\left[\left(\frac{k}{\ell}\right)^2r_2-\delta_4\right]\left(\frac{k}{\ell}\right)^2r_1-\left(\frac{k}{\ell}\right)^2\delta_1r_2+\delta_1\delta_4-\delta_2\delta_3.$$

If $\delta_4\ell^2\leq r_2$, then $\mathcal{D}_k>0$ for all $k\in\mathbb{Z}^+$, and the steady state $E_2$ is stable.

If $\delta_4\ell^2> r_2$, solve $\mathcal{D}_k=0$ for $r_1$ and obtain
\begin{equation}\label{r1}
r_1=r_1^{(k)}:=\frac{(\delta_1\delta_4-\delta_2 \delta_3 -\delta_1 r_2 \frac{k^2}{\ell^2} )}{\frac{k^2}{\ell^2}(\delta_4-r_2 \frac{k^2}{\ell^2})},
\end{equation}
where the superscript $(k)$ emphasizes the dependence of $r_1$ on the integer $k$. Notice that the numerator of $r_1^{(k)}$ is positive, so $r_1^{(k)}>0$ only if $0<k\leq\bar{k}$, where \begin{equation*}
    \bar{k}=\begin{cases}
     \sqrt{\frac{\delta_4\ell^2}{r_2}}-1 \hskip 1cm  \text{if}\,  \sqrt{\frac{\delta_4\ell^2}{r_2}}\in \mathbb{Z}^+,\\
\lfloor\sqrt{\frac{\delta_4\ell^2}{r_2}}\rfloor  \hskip 1.32cm  \text{if}\,  \sqrt{\frac{\delta_4\ell^2}{r_2}}\notin \mathbb{Z}^+,\\
\end{cases}
\end{equation*}
where $\lfloor\cdot\rfloor$ denotes the floor function.  It follows from
 $$\frac{d}{d k}\left(r_1^{(k)}\right)=-\frac{2\left[\left(\delta_1 \,\delta_4-\delta_2 \delta_3 \right)\delta_4\ell^4-2 k^2 r_2 \left(\delta_1 \delta_4 -\delta_2 \delta_3 \right) \ell^2 +\delta_1 \,k^4 r_2^{2}\right] \ell^2}{k^3 \left(\delta_4 \ell^2 -k^2 r_2 \right)^{2}}$$ that $r=r_1^{(k)}$ has a critical point
$$k=\hat{k}=\sqrt{\left(\delta_1 \delta_4 -\delta_2 \delta_3 +\sqrt{-\delta_3\delta_2(\delta_1\delta_4-\delta_3\delta_2)}\right) \ell^2/(\delta_1 r_2)}\in (0, \sqrt{\delta_4\ell^2/r_2}).$$
 Moreover, $\frac{d}{d k}(r_1^{(k)})<0$ for $k<\hat{k}$ and  $\frac{d}{d k}(r_1^{(k)})>0$ for $k>\hat{k}$. Hence, the function $r_1=r_1^{(k)}$ is $U$-shaped in the first quadrant of the $(k, r_1)$-plane. It is apparent that $\mathcal{D}_k>0$ if $r_1<r_1^{(k)}$, and $\mathcal{D}_k<0$ if $r_1>r_1^{(k)}$.

Define $\mathscr{N}=\{k\in\mathbb{N}|1\leq k\leq\bar{k}\}$, and take
$$\breve{k}=\arg\min_{k \in \mathscr{N}} r_1^{(k)},$$ i.e.,  the function $r_1=r_1^{(k)}$ is minimized at $k=\breve{k}$.

\begin{figure}[H]
    \centering    \includegraphics[width=0.6\linewidth]{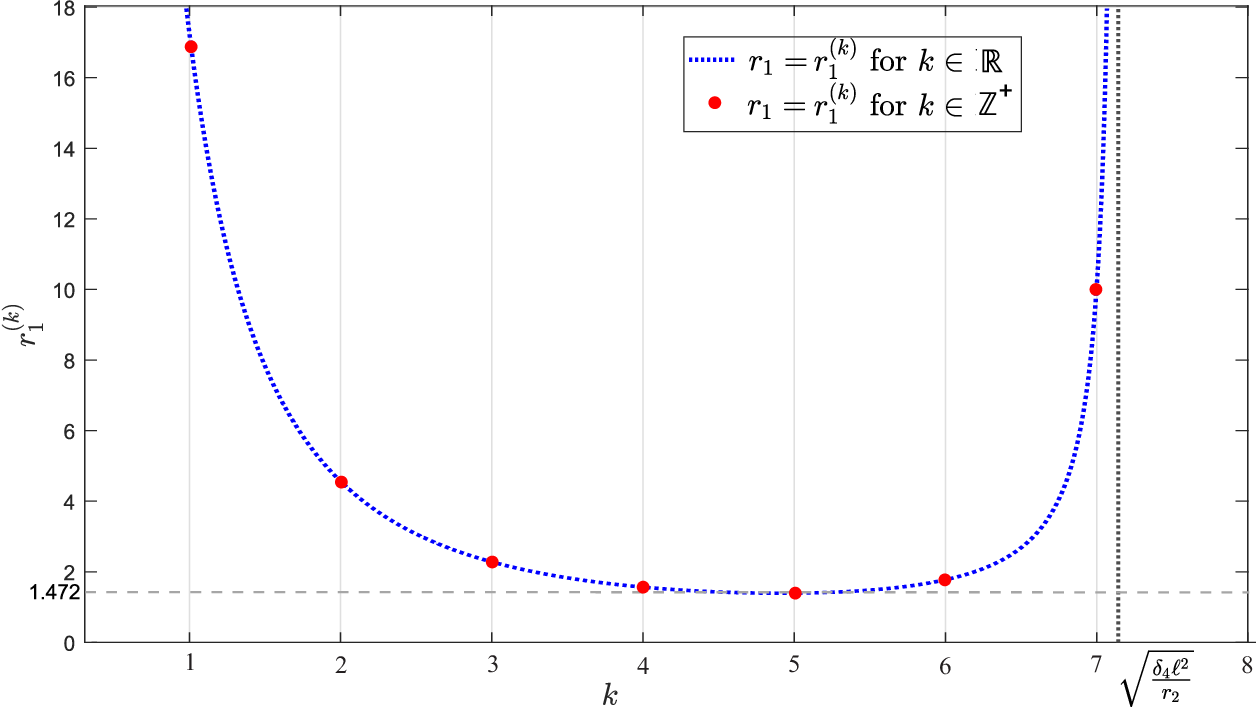}
    \caption{\footnotesize Graph of $r_1=r_1^{(k)}$.  Here, $\beta = 0.1$,
$d = 0.1$,
$\mu_0 = 0.1$,
$\mu_1 = 0.2$,
$b = 0.3$,
$A = 0.4$,
$\ell=5$,
$r_2 = 0.01$.
The blue dashed curve represents the overall profile of \(r=r_1^{(k)}\) as \(k\) varies continuously. The red markers pinpoint the discrete \(k\)-mode Turing bifurcations, where the vertical lines at \(k=1,\cdots,7\) meet the \(r_1^{(k)}\) curve. The graph shows the presence of 7 distinct $k$-mode Turing bifurcations at approximately $(1,17.045),(2,4.6028),(3,2.346),(4,1.638),(5,1.472),(6,1.859),(7,10.733)$. From the graph, the upper bound for $k$ is $\bar{k}=7$, and the lowest threshold of Turing bifurcation is $r_1=1.471$, which occurs at $k=\Breve{k}=5$.}
\label{Fig: r1vsk}
\end{figure}

\begin{theorem}\label{k-mode Turing theorem}
    Under the stability condition defined in Theorem \ref{Tr(J)<0},  the following holds.
    \begin{enumerate}\setlength\itemsep{0.cm}
        \item[(1)] The steady state $E_2$ is locally asymptotically stable if  one of the following is satisfied:
        \begin{enumerate}\setlength\itemsep{0.cm}
            \item $\gamma>\gamma_-$;

            \item $\gamma<\gamma_-$ and $\delta_4\ell^2\leq r_2$;

            \item $\gamma<\gamma_-, \delta_4\ell^2>r_2$ and $ r_1<r_1^{(\breve{k})}$.
        \end{enumerate}

        \item[(2)] The steady state $E_2$ is locally unstable for $\gamma<\gamma_-,\delta_4\ell^2>r_2$ and $r_1>r_1^{(\breve{k})}$.

        \item[(3)] Suppose that $r_1^{(i)}\neq r_1^{(j)}$ for $\forall i, j\in \mathscr{N}$ and $i\neq j$. Then system \eqref{RD model} undergoes a $k$-mode Turing bifurcation for $\gamma<\gamma_-,\delta_4\ell^2>r_2$ and $r_1=r_1^{(k)}$, where $k\in\mathscr{N}$.
    \end{enumerate}
\end{theorem}

\begin{proof}



Statements (1) and (2) are obvious. We only prove statement (3).

Fix $k\in\mathscr{N}$. If $r_1=r_1^{(k)}$, If follows from the assumption that the matrix $J_n$ $(\forall  n\in\mathbb{Z}^+\setminus\{k\})$ has eigenvalues with negative real parts, while $J_{k}$ has a negative eigenvalue and a zero eigenvalue. A straightforward calculation leads to
$$\frac{d\mathcal{D}_k}{dr_1} _{\big|r_1=r_1^{(k)}}=\frac{\left(r_2 k^2-\delta_4 \ell^2\right)k^2}{\ell^4}<0.$$
The inequality holds since $k\leq\bar{k}$. Hence, a single real eigenvalue of $J_{k}$ crosses the imaginary axis from left to right transversely at the origin, and a $k$-mode Turing bifurcation occurs at $r_1=r_1^{(k)}$.
\end{proof}

\begin{remark} For the numerical study presented in this paper,  we assume the initial condition $S(x, 0)=S_2+0.01\cos(0.4x), I(x, 0) = I_2 + 0.01\cos(0.4x), x\in [0, 5\pi]$ unless otherwise stated.
\end{remark}

\begin{figure}[!ht]
    \centering
    \subfigure[Surface $S(x, t)$]{%
        \includegraphics[width=0.45\linewidth]{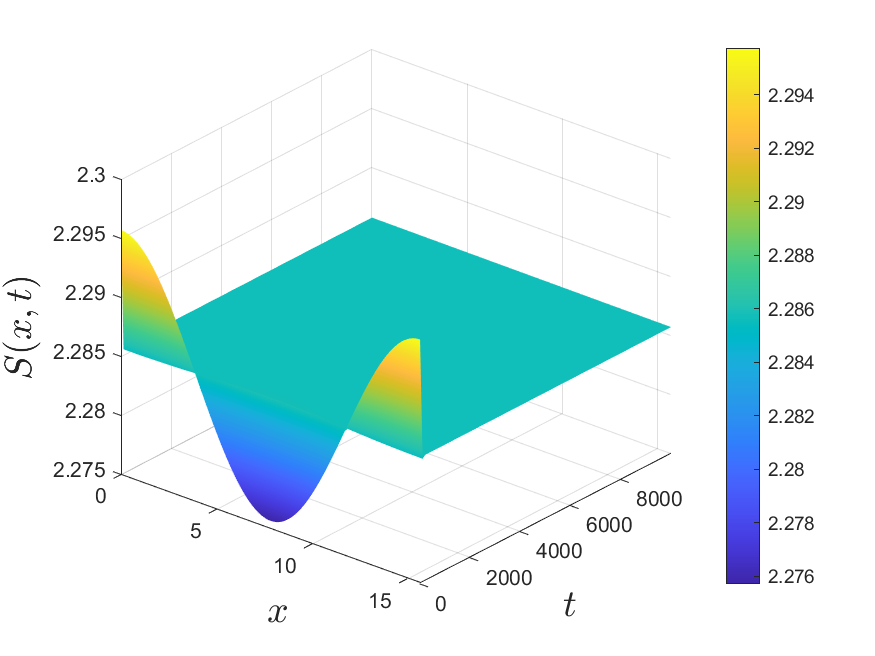}%
    }
    \hfill
    \subfigure[Surface $I(x, t)$]{%
        \includegraphics[width=0.45\linewidth]{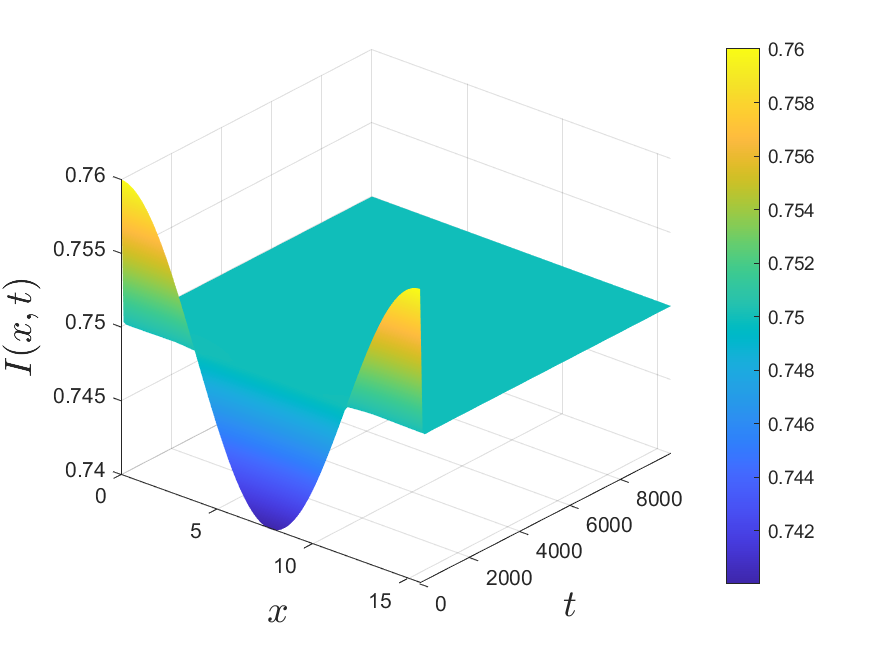}%
    }
    \caption{The solution converges to the stable constant steady state $E_2$.}  \label{steady}
\end{figure}

\begin{figure}[!ht]
    \centering
    \subfigure[Surface $S(x, t)$]{%
        \includegraphics[width=0.45\linewidth]{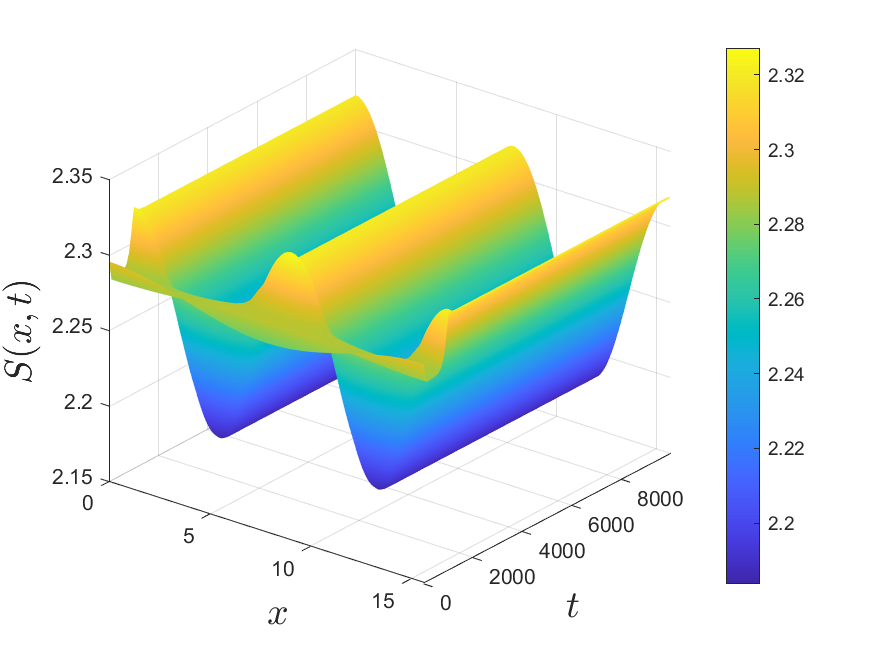}%
    }
    \hfill
    \subfigure[Surface $I(x, t)$]{%
        \includegraphics[width=0.45\linewidth]{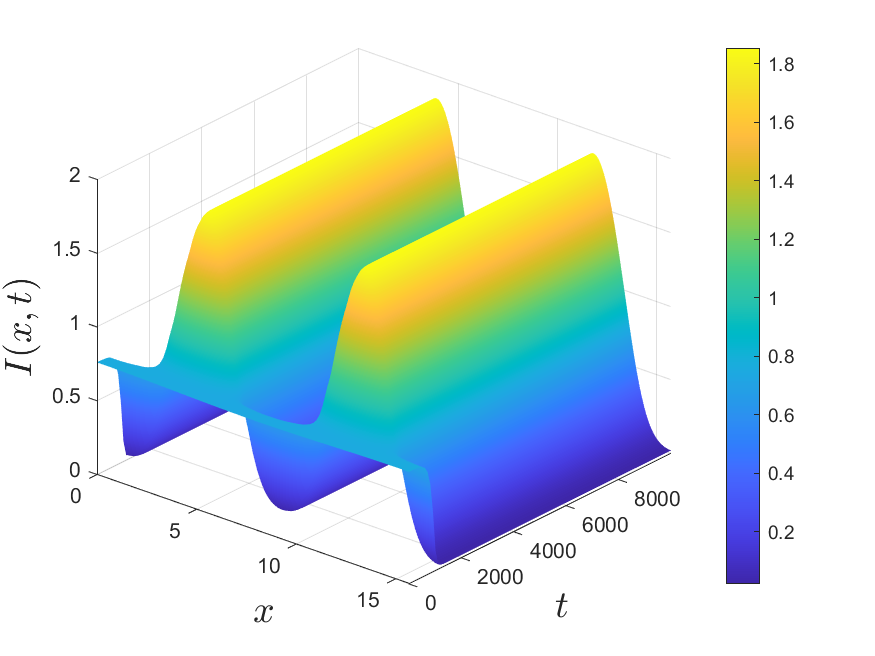}%
    }
    \caption{A spatially inhomogeneous solution (spatially patterned wave) merges due to a sequence of $k$-mode Turing bifurcation ($k=5, 4, 6, 3$), while the constant steady state $E_2$ loses its stability.} \label{Turing}
\end{figure}

From Fig. \ref{Fig: r1vsk} one can see that the minimal threshold of Turing bifurcation for $r_1$ is $r_1^{(5)}=1.472$. If $r_1=1<r_1^{(5)}$, by Theorem \ref{k-mode Turing theorem} (1), diffusion does not destabilize the constant steady state $E_2$. See Fig. \ref{steady} (a) and (b), where the solution converges to the stable steady state $E_2$ as the initial perturbation damps rapidly. If $r_1=4.6028$, then $r_1>r_1^{(5)}> r_1^{(4)}>r_1^{(6)}>r_1^{(3)}$. By Theorem \ref{k-mode Turing theorem} (3), a sequence of $k$-mode Turing bifurcation occurs ($k=5, 4, 6, 3$) as $r_1$ increases from 1 to 1.472, and the formation of spatially inhomogeneous solution is across the domain. See  Fig. \ref{Turing} (a) and (b), from which one can observe that the surfaces of $S(x, t)$ and $I(x, t)$ are striped in the $x$ direction, yet flat in the $t$ direction. 



\subsection{\texorpdfstring{$k$}{}-mode Hopf bifurcation}

\begin{definition}
\cite{jiang2020formulation}
Let $k\in \mathbb{Z}^+\cup\{0\}$ be an integer and fixed. Let $\xi\in\mathbb{R}$ be a bifurcation parameter. A $k$-mode Hopf bifurcation occurs if the Jacobian $J_k$ has a pair of purely imaginary eigenvalues crossing the imaginary axis transversally and all eigenvalues of $J_n (n\neq k)$ are not on the imaginary axis as $\xi$ passes through $\xi=\xi^\ast\in \mathbb{R}$.
\end{definition}

Consider $\beta$ as a bifurcation parameter, and regard $\mathcal{T}_k$ as a function of $\beta$, i.e., $\mathcal{T}_k=\mathcal{T}_k(\beta)$, then we have the $k$-mode Hopf bifurcation theorem as follows.

\begin{theorem}\label{kHopf}
Let $k\in \mathbb{Z}^+\cup\{0\}$ be fixed. If there exists $\beta=\beta_H^{(k)}$ such that $\mathcal{T}_k(\beta_H^{(k)})=0$ and $\mathcal{T}'_k(\beta_H^{(k)})\neq 0$. Then system \eqref{RD model} undergoes a $k$-mode Hopf bifurcation around the steady state $E_2$ at $\beta=\beta_H^{(k)}$ provided that $\mathcal{D}_k(\beta_H^{(k)})>0$.
\end{theorem}

\begin{proof}
Since $\mathcal{T}_k(\beta_H^{(k)})=0$, then $J_k$ has a pair of complex conjugate eigenvalues $\lambda_{1,2}=\lambda_{1,2}(\beta)$ in a small neighborhood of $\beta_H^{(k)}$ and $Re(\lambda_{1,2}(\beta_H^{(k)}))=0$ and $Im(\lambda_{1,2}(\beta_H^{(k)}))\neq 0$. Then
$$\dfrac{d(Re(\lambda_{1,2}(\beta)))}{d\beta}_{|\beta=\beta_H^{(k)}}=\frac{1}{2}\mathcal{T}'_k(\beta_H^{(k)})\neq 0.$$ Hence, the transversality condition is fulfilled.

Note that $\mathcal{T}_j=\mathcal{T}_k+\frac{(r_1+r_2)(k^2-j^2)}{\ell^2}$. Then  $\mathcal{T}_j(\beta_H^{(k)})=\frac{(r_1+r_2)(k^2-j^2)}{\ell^2}\neq 0$ for $j\neq k$. Therefore, all eigenvalues of $J_n (n\neq k)$ have no zero real parts. The desired result is obtained.
\end{proof}

\begin{remark} For analyzing the stability of the periodic solution bifurcating from the $k$-mode Hopf bifurcation, it is necessary to compute the first Lyapunov coefficient. It has been shown that the first Lyapunov coefficient of the $0$-mode Hopf bifurcation of the PDE model has the same sign as that of the underlying ODE model \cite{SR2015, YWS2009}. Hence, the stability and direction of the $0$-mode Hopf bifurcation have been characterized by Theorem \ref{Hopf}. In contrast, for the stability and direction of the $k$-mode Hopf bifurcation with $k\in\mathbb{Z}^+$, we need to apply the center manifold theorem and normal form computation, and we will leave it for future study.
\end{remark}

\begin{figure}[H]
    \centering
    \subfigure[Surface $S(x, t)$]{%
         \includegraphics[width=0.48\linewidth]{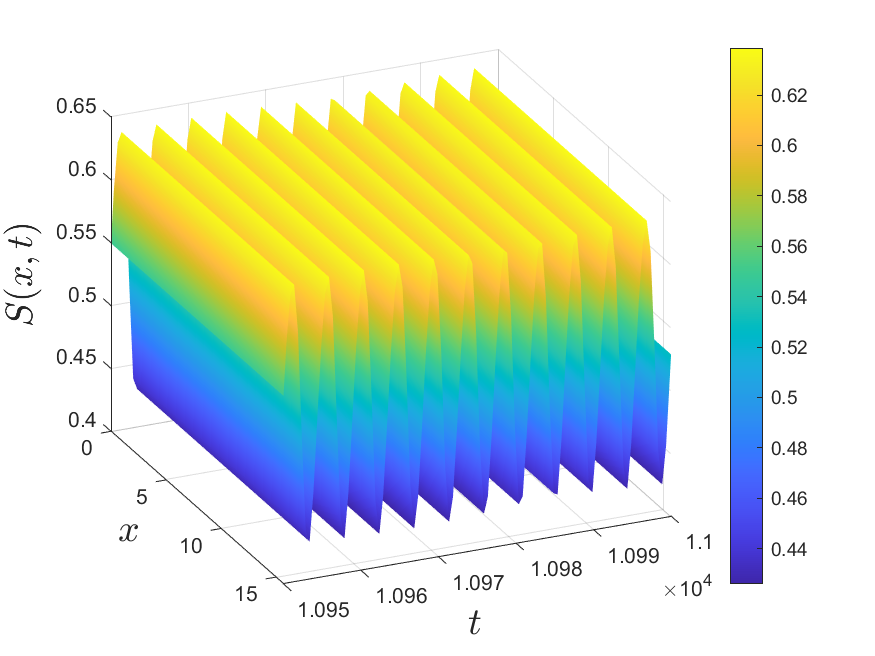}%
    }
    \hfill
    \subfigure[Surface $I(x, t)$]{%
         \includegraphics[width=0.48\linewidth]{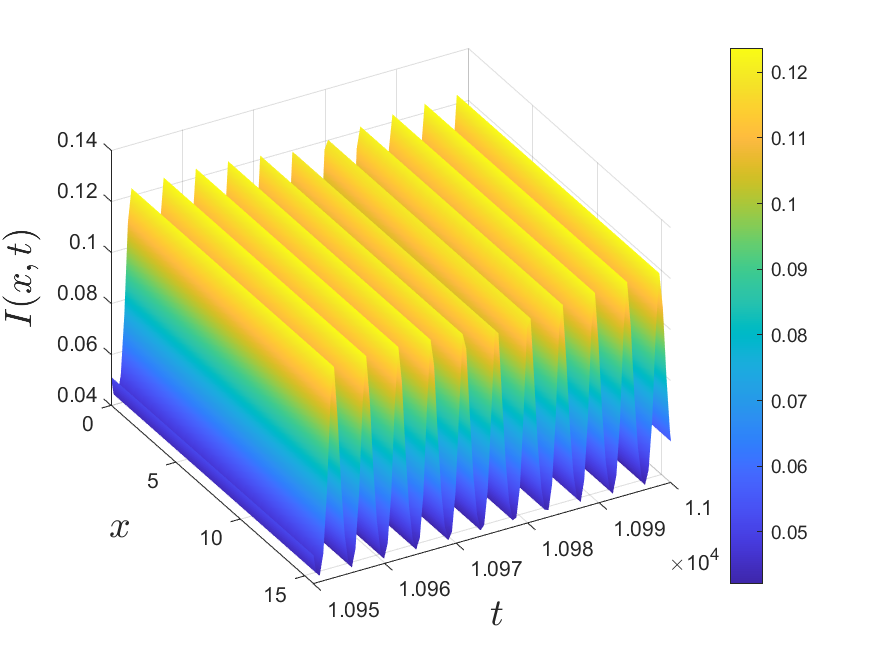}%
    }
    \caption{A spatially homogeneous periodic solution generated by $0$-mode
Hopf bifurcation. It is the spatiotemporal version of the limit cycle in Fig. \ref{HH} (b).}
    \label{zeroHopf}
\end{figure}

\begin{remark} The periodic solutions bifurcating from 0-mode Hopf bifurcation are spatially homogeneous, which coincides with the periodic solution of the underlying ODE system \eqref{submodel}. The periodic solutions bifurcating from a $k$-mode Hopf bifurcation are spatially inhomogeneous if $k\in\mathbb{Z}^+$.
\end{remark}

\noindent {\bf Example 2}. {\it  Spatially homogeneous and inhomogeneous periodic solution.}


\begin{itemize}\setlength\itemsep{0.1cm}
    \item[(1)] The model \eqref{RD model} exhibits a spatially homogeneous periodic solution generated by $0$-mode Hopf bifurcation.  See Fig. \ref{zeroHopf}, where $r_1=0.01$, $r_2=0.01$, and other parameters are the same as those for Fig. \ref{HH} (b). We obtain the exact limit cycle in Fig. \ref{HH} (b) by fixing any $x\in\Omega$. The surfaces of $S(x, t)$ and $I(x, t)$ show the periodicity in the $t$ direction and a series of parallel ridges aligned along the $x$-axis. The solutions are constant in the $x$ direction, implying spatial homogeneity and temporal oscillation.

    \item[(2)] The model \eqref{RD model} exhibits a spatially inhomogeneous periodic solution near a $2$-mode Hopf bifurcation. See Fig. \ref{kkHopf}. The surfaces of $S(x, t)$ and $I(x, t)$ show striped, ridge-like structures, indicating that infective clusters persist and reappear in fixed spatial locations, though the intensity varies over time.
Here, $\beta = 0.0088$, $d = 0.01$, $\mu_0 = 0.1$, $\mu_1 = 10$, $b = 0.03$,
$A=1$, $r_1=0.05$ and $r_2 = 0.01$. The initial condition is $S(x, 0)=S_2+0.00001\cos(0.2x), I(x, 0)=I_2 + 0.00001\cos(0.2x), x\in [0, 5\pi]$.
\end{itemize}

\begin{figure}[H]
    \centering
    \subfigure[Surface $S(x, t)$]{%
         \includegraphics[width=0.48\linewidth]{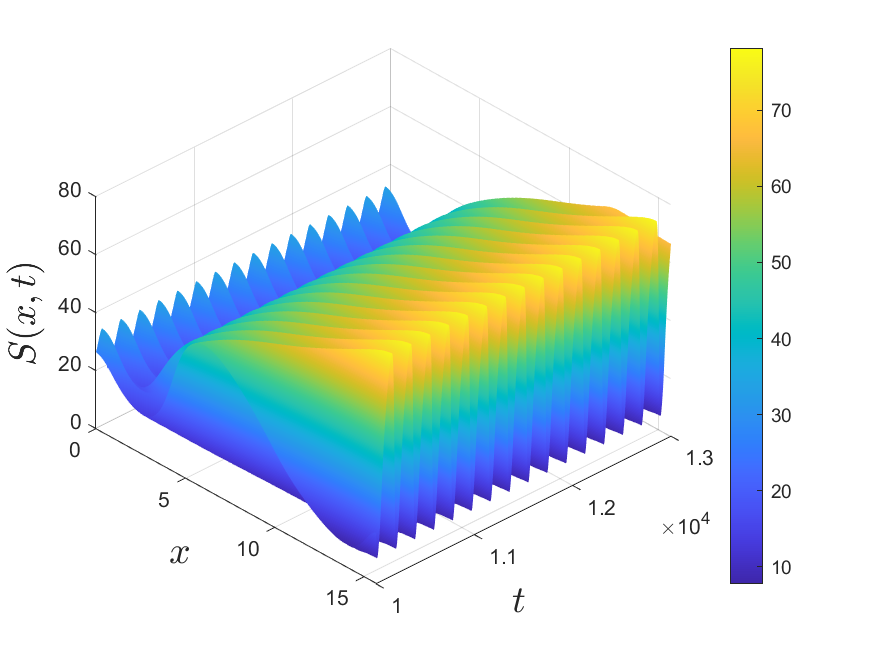}%
    }
    \hfill
    \subfigure[Surface $I(x, t)$]{%
         \includegraphics[width=0.48\linewidth]{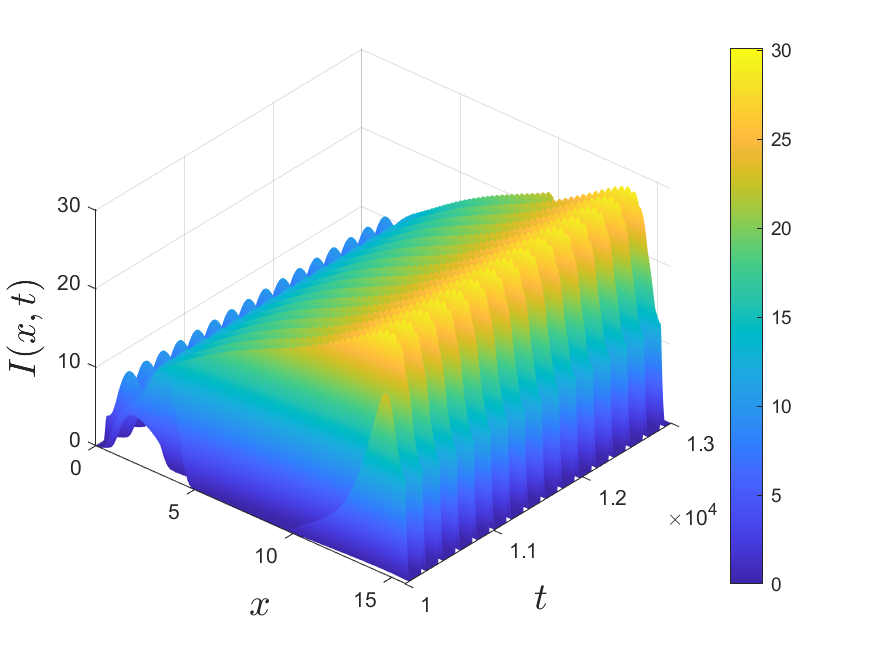}%
    }
    \caption{A spatially inhomogeneous periodic solution near a $2$-mode Hopf bifurcation.}
    \label{kkHopf}
\end{figure}


\subsection{\texorpdfstring{$(k_1, k_2)$}{}-mode Turing-Hopf bifurcation}




The occurrence of Turing-Hopf bifurcation is one of the main mechanisms for the spatiotemporal patterns of a reaction-diffusion system, which has been extensively observed and studied \cite{AnJiang2019, CJ2022, JiangWangCao2019, LWD, RMP, RovinskyMenzinger1992, SanchezGarduno2019, SR2015,  SJLY, turing1952morphogenesis, WangWangQi2024,YangSong2016}. The terminology and definition of a $(k_1, k_2)$-mode Turing-Hopf bifurcation are proposed by Jiang et al. The readers are referred to \cite{jiang2020formulation} for the definition.  Based on The Turing instability Theorem \ref{k-mode Turing theorem}, Hopf bifurcation Theorem \ref{kHopf}, and according to the definition of Turing-Hopf bifurcation \cite{jiang2020formulation}, we have the results if the Turing bifurcation and Hopf bifurcation occur simultaneously.

\begin{theorem}\label{Turing-Hopf}  Suppose that $r_1^{(i)}\neq r_1^{(j)}$ for $\forall i, j\in \mathscr{N}$ and $i\neq j$. Let $k_1\in\mathbb{Z}^+$, $k_2\in\mathbb{Z}^+\cup \{0\}$ and $k_1\neq k_2$. Suppose that  $r_1=r_1^{(k_1)}$ and there exists $\beta=\beta_H^{(k_2)}$ such that $\mathcal{T}_{k_{2}}(\beta_H^{(k_2)})=0$ and $\mathcal{T}'_{k_{2}}(\beta_H^{(k_2)})\neq 0$.  Then system $\eqref{RD model}$ undergoes a $(k_1,k_2)$-mode Turing-Hopf bifurcation around the steady state $E_2$ for $(r_1, \beta)=(r_1^{(k_1)}, \beta_H^{(k_2)})$ provided $0\leq k_2<k_1\leq \bar{k}$ and $k_2$ satisfies condition ($\mathbb{H}$) defined by \eqref{condition}.
\end{theorem}



\begin{proof}
To obtain Turing instability, the trace of $J_{k} (k\in\mathscr{N})$  should remain negative when its determinant is set equal to zero. Since $\mathcal{T}_k$ is decreasing in $k$ and the Hopf bifurcation occurs when the trace is equal to zero, it is necessary to have $0\leq k_2<k_1\leq \bar{k}$.

By implementing the argument of Theorem $\ref{k-mode Turing theorem},$ we have $\mathcal{D}_{k_1}=0$ for any $k_1\leq \bar{k}$, $r_1=r_1^{(k_1)}\geq r_1^{\breve{(k)}}$ and $\delta_4\ell^2>r_2$. Hence,  $J_{n}$ has no zero eigenvalue for all $n\in\mathbb{Z}^+\cup \{0\}\setminus\{k\}$. Due to the $U$-shape of function $r_1=r_1^{(k)}$, $\mathcal{D}_k$ changes signs from positive to negative then to positive as $k$ increases throughout the interval $(0, \infty)$, which occurs at $k=k_1\in\mathbb{Z}^+$, and at the complementary root $k=k^\ast\in\mathbb{R}^+$, where
$$k^\ast=\frac{\sqrt{r_2\left(\delta_1\delta_4\ell^{2}-\delta_1k_1^{2}r_2-\delta_2\delta_3\ell^{2}\right)\left(\delta_1\delta_4^{2}\ell^{2}-\delta_1\delta_4k_1^{2}r_2-\delta_2\delta_3\delta_4\ell^{2}+\delta_2\delta_3k_1^{2}r_2\right)}\ell}{r_2\left(\delta_1 \delta_4\ell^{2}-\delta_1k_1^{2}r_2-\delta_2\delta_3\ell^{2}\right)}.$$


For $\beta=\beta_H^{(k_2)}$, it follows from $\mathcal{T}_{k_{2}}(\beta_H^{(k_2)})=0$ and $\mathcal{T}'_{k_{2}}(\beta_H^{(k_2)})\neq 0$ that $J_{k_2}$ exhibits a pair of pure imaginary eigenvalues satisfying the transversality condition provided $\mathcal{D}_{k_2}$ is strictly positive, which is feasible if one of the following conditions holds:
\begin{equation}
   \mathbb{H}: \begin{cases}
     \mathbb{H}_1: k_2< k_1\leq\breve{k}\leq\lfloor k^*\rfloor,\\
  \mathbb{H}_2: k_2 \leq \lfloor k^*\rfloor\leq\breve{k}\leq k_1.
\end{cases}\label{condition}
\end{equation}
If ($\mathbb{H}_1$) holds, we have $\mathcal{D}_{k_1}=0$, $\mathcal{D}_{k_2}>0$, and $\frac{d\mathcal{D}_k}{dk}|_{k=k_1,r=r_1^{(k_1)}}<0$. If ($\mathbb{H}_2$) holds, we have $\mathcal{D}_{k_1}=0$, $\mathcal{D}_{k_2}>0$, and $\frac{d\mathcal{D}_k}{dk}|_{k=k_1,r=r_1^{(k_1)}}>0$.  The condition ($\mathbb{H}$) is sufficient for the positiveness of $\mathcal{D}_{k_2}$, and consequently the existence of the pure imaginary conjugate eigenvalues. This finishes the proof.
\end{proof}

    \begin{figure}[H]
    \centering
    \includegraphics[width=0.6\linewidth]{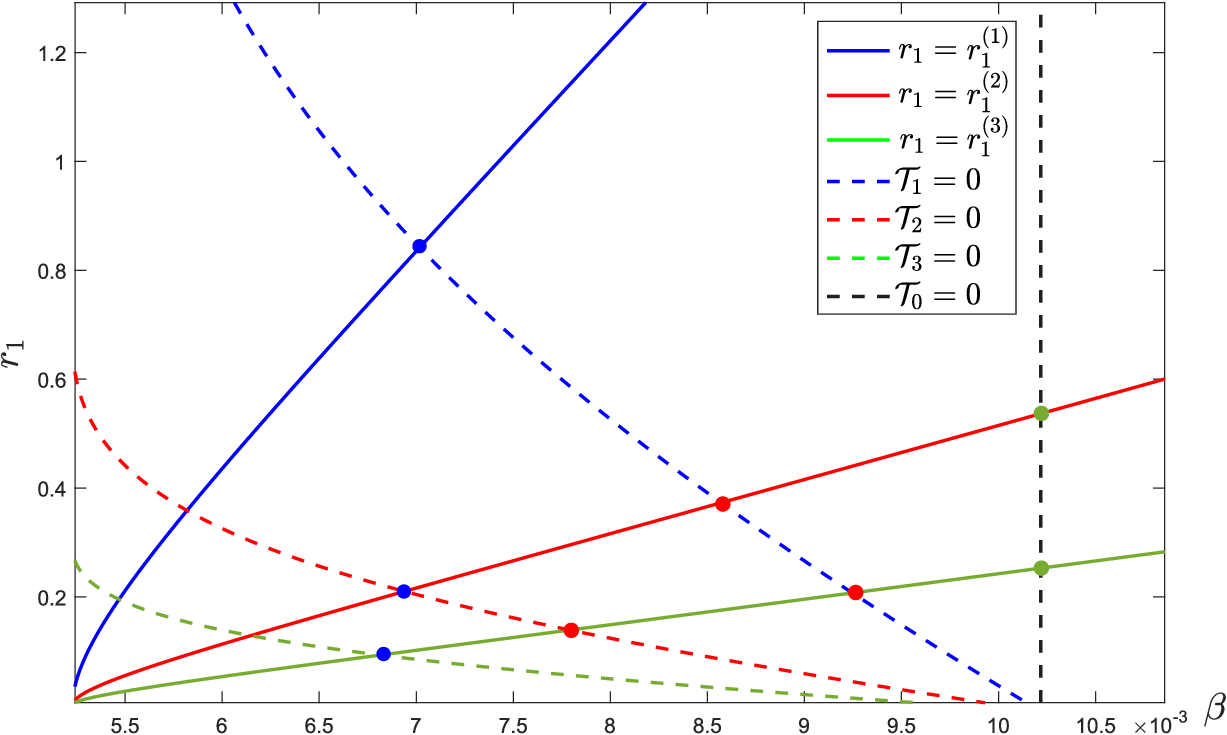}
    \caption{ \footnotesize   The solid lines represent the $k$-mode Turing bifurcation curves $r_1=r_1^{(k)}$, while the dashed lines are the curves of $\mathcal{T}_k(r_1,\beta)=0$. The red points mark the parameter values at which $(k_1,k_2)$-mode Turing-Hopf bifurcation occurs ($k_2=1, 2$ and $k_2<k_1$). Blue points indicate where both $\mathcal{T}_k$ and $\mathcal{D}_k$ are simultaneously zero for the same $k$-mode ($k=1, 2, 3$), resulting in a double zero bifurcation \cite{CJ2020}. The vertical line denotes the $0$-mode Hopf bifurcation. The green points indicate the parameters for which system \eqref{RD model} undergoes a $(k,0)$-mode Turing-Hopf bifurcation ($k=2, 3$). For the sake of the readability of the graph, we limited $k$ to 3.}
    \label{fig: r1 vs beta}
\end{figure}

\noindent {\bf Example 3}. Choose $d=0.01$, $\mu_0=0.1$, $\mu_1=10$, $b=0.03$,  $A=1$ and $r_2=0.01$.

(1) The Turing-Hopf bifurcation points can be observed in Fig. \ref{fig: r1 vs beta}, where the red and green points mark the values at which the $(k_1, k_2)$-mode Turing-Hopf bifurcation occurs. Particularly the $(4, 3)$-mode Turing-Hopf bifurcation occurs at we have
$(r_1^{(4)}, \beta_H^{(3)})\approx (0.0721, 0.0073)$.

(2) We choose $(r_1, \beta)=(0.07208, 0.00734)$ sufficiently closed to $(r_1^{(4)}, \beta_H^{(3)})$, and sketch the solution $S(x, t)$ and $I(x, t)$, which exhibit a spatiotemporal pattern near the $(4, 3)$-mode Turing-Hopf bifurcation point. See Fig. \ref{Turing-Hopf figure} for the surfaces of $S(x, t)$ and $I(x,t)$, and Fig. \ref{Turing-Hopf-Projection} for the heatmap.

\begin{figure}[H]
    \centering
    \subfigure[Surface $S(x, t)$]{%
         \includegraphics[width=0.45\linewidth]{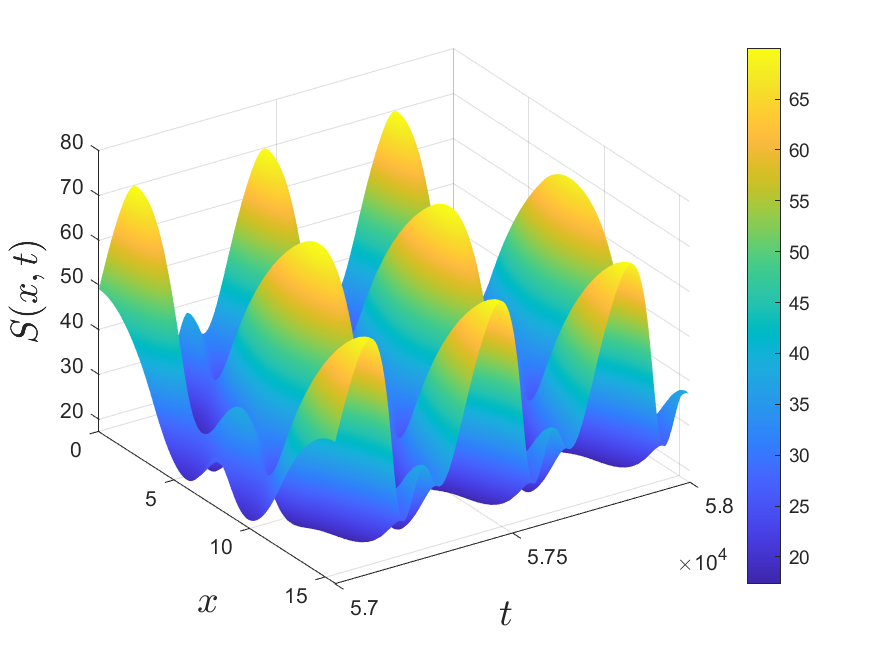}%
    }
    \hfill
    \subfigure[Surface $I(I, t)$]{
         \includegraphics[width=0.45\linewidth]{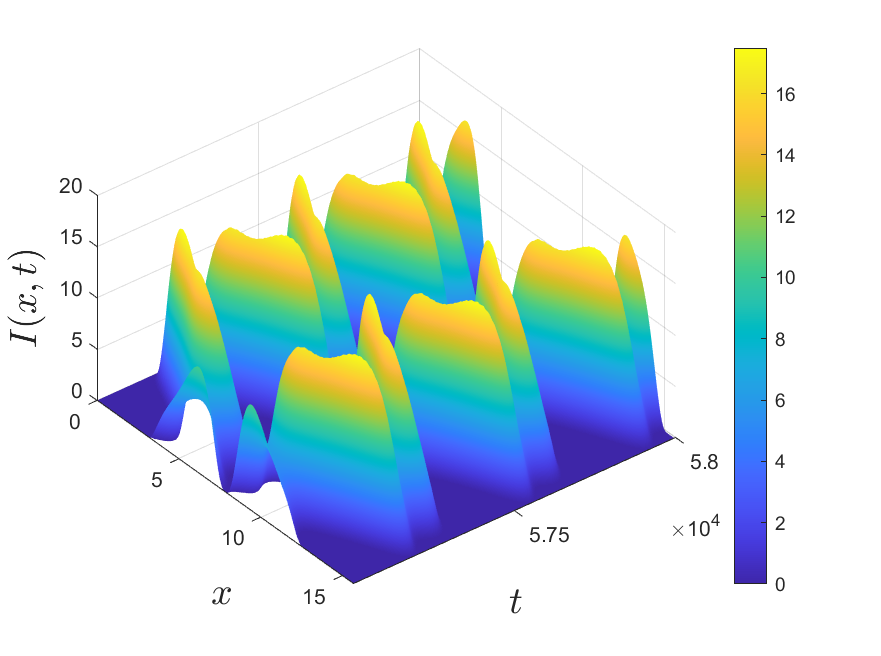}%
    }
    \caption{\footnotesize The long-term spatiotemporal dynamics of the susceptible and infected populations on the domain $x\in[0,5\pi]$ near the $(4,3)$-mode Turing-Hopf bifurcation point. Figures (a) and (b) resemble the 3D surfaces of $S(x,t)$ and $I(x,t)$. The susceptible and infected populations oscillate in both space and time, forming peaks across the domain that repeat periodically with a fixed period. Regional pulses with amplitudes ranging from 0 to roughly 20 that follow the troughs of susceptible population $(S)$ are visible in the infected population $(I)$. }
    \label{Turing-Hopf figure}
\end{figure}

\begin{figure}[H]
    \centering
 \subfigure[Projection of $S(x, t)$]{%
         \includegraphics[width=0.45\linewidth]{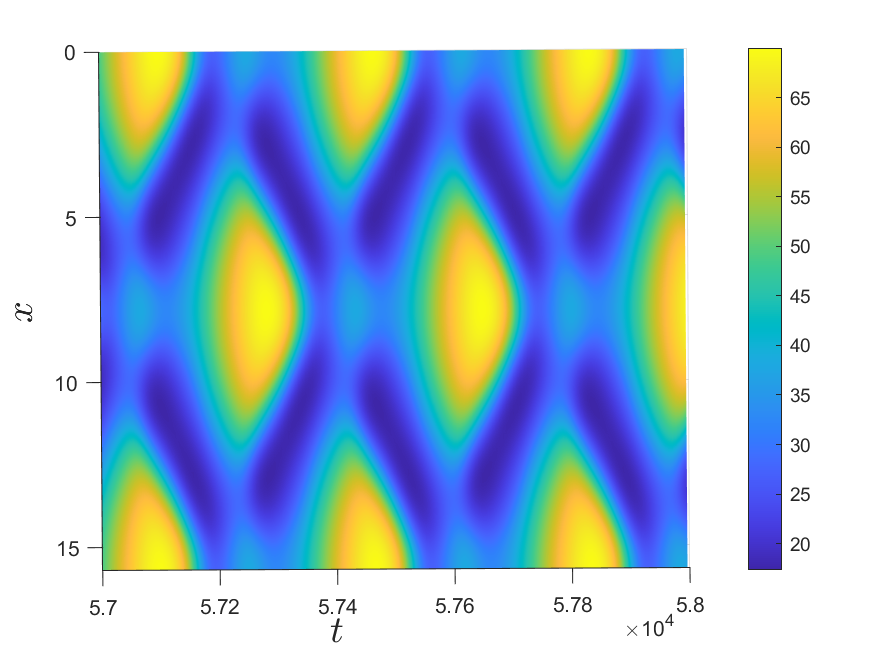}%
    }
    \hfill
    \subfigure[Projection of $I(x, t)$]{
         \includegraphics[width=0.45\linewidth]{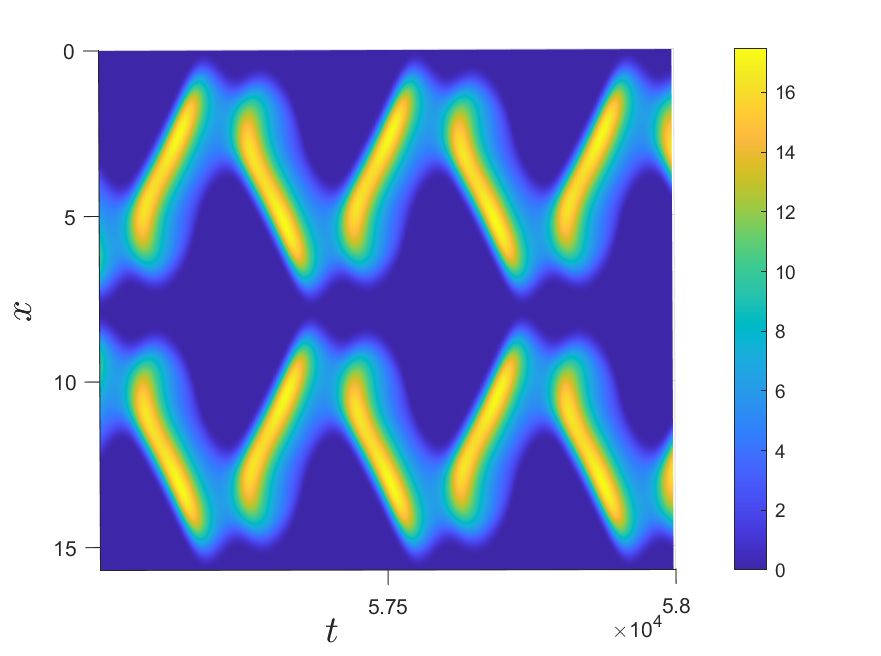}%
      }
    \caption{\footnotesize Projection of the spatiotemporal pattern of Fig. \ref{Turing-Hopf figure} onto $(x, t)$-plane. The heatmap illustrates the spatial pattern and temporal oscillation: cool (blue) regions correspond to low $(S)$ and $(I)$ while warm (yellow) regions correspond to high $(S)$ and $(I)$. The infected peaks show up as bright regions, suggesting that a wave-like pattern of infection outbreak travels through space.}
    \label{Turing-Hopf-Projection}
\end{figure}


\section{Conclusion}

Based on the theoretical analysis on the reaction kinetics and pattern formation of the diffusive model \eqref{RD model} in Sections 2 and 3, there are several aspects we would like to address as follows.

{\bf (a) Activator and inhibitor of pattern formation}. We have proven that the diffusion of the susceptible and infected classes will not change the stability of the constant steady states $E_0$ and $E_1$ (see Theorem \ref{T31} and Remark \ref{R32}). Hence, the spatial pattern of the reaction-diffusion system driven by diffusion occurs only when the constant steady state $E_2$ loses its stability. We see that the faster movement of the susceptible class will induce the spatial and spatiotemporal patterns, which are characterized by $k$-mode Turing instability. Moreover, the infected population functions as an activator, while the susceptible population serves as an inhibitor (see Theorem \ref{kHopf} and Theorem \ref{Turing-Hopf}).

To illustrate the idea, we sketch the Turing bifurcation curve in the $(r_1, r_2)$-plane. See Fig. \ref{r1r2graph}. The lines $r_1=r_1^{(k)}$ denote the curves of $k$-mode Turing bifurcation. One can see that
all Turing bifurcation curves lie above the dashed blue line  $r_1=r_2/\gamma_{-}$, where $\gamma_{-}$ signifies the threshold ratio for the diffusion rates between the infected class and susceptible class. The Turing bifurcation occurs only when the diffusion rate $r_1$ of the susceptible class exceeds $r_2/\gamma_{-}$. The black dots indicate the intersections between each pair of curves of Turing bifurcation, corresponding to Turing–Turing bifurcation points \cite{CJ2022, ZW2022}. The occurrence of Turing–Turing bifurcation implies the superposition of the two spatial patterns.

\begin{figure}[!ht]
    \centering
    \includegraphics[width=0.65\linewidth]{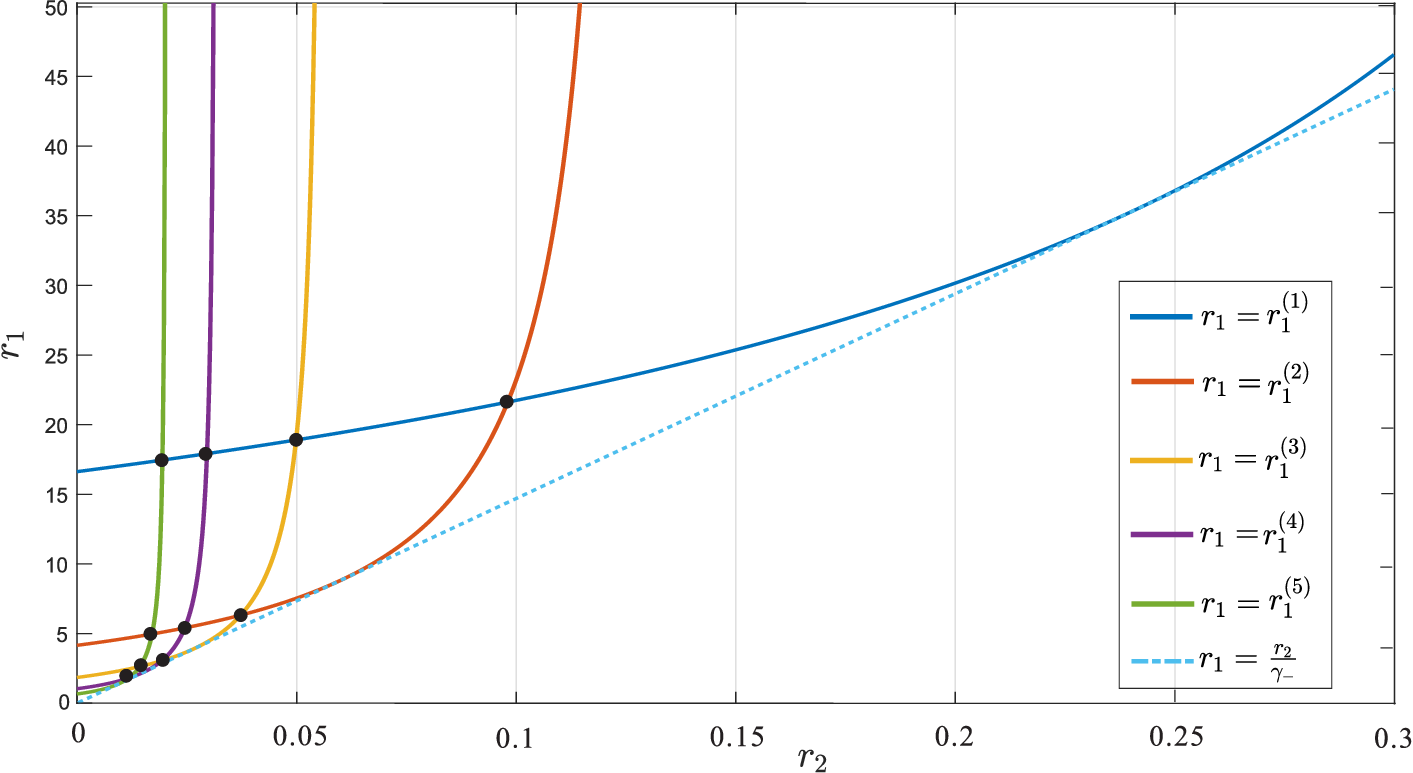}
    \caption{\footnotesize Curves of $k$-mode Turing bifurcation $r_1=_1^{(k)}$ lie above and are tangent to the line $r_1=r_2/\gamma_-$ in $(r_1, r_2)$ plane. The black dots indicate Turing–Turing bifurcation points. Here, all parameters are the same as those for Fig. \ref{Fig: r1vsk} except the free diffusion rates $r_1$ and $r_2$.}
    \label{r1r2graph}
\end{figure}

{\bf (b) Two types of temporal oscillations}. The periodic solutions merging from 0-mode Hopf bifurcation are spatially homogeneous, while the periodic solutions merging from a $k$-mode Hopf bifurcation are spatially inhomogeneous if $k\in\mathbb{Z}^+$, which characterizes the scenario that infected and susceptible populations oscillate over time and vary from place to place in a non-uniform but regular way.

{\bf (c) Transient dynamics from temporal oscillatory regime to a spatial periodic pattern}. It is clearly shown in Fig. \ref{fig: r1 vs beta} that there are several $(k_1, k_2)$-mode Turing-Hopf bifurcation points and double-zero bifurcation points as $(r_1, \beta)$ varies. Hence, the spatiotemporal dynamics are sensitive to bifurcation parameters $(\beta, r_1)$ and the initial condition. We fix $(\beta, r_1)=(0.0094, 0.07)$, and plot the surfaces and the heatmap of the solution $S(x, t)$ and $I(x, t)$ in Fig. \ref{surface_transient} and Fig. \ref{projection_transient}, which suggest a dynamical transition occurring near $t\approx 2800$ beyond which spatial structure becomes pronounced. From an epidemiological perspective, these figures illustrate the dynamic evolution of an epidemic in a spatial domain, highlighting a critical qualitative shift in how the disease spreads:

{\it Early Stage ($t<2800$)}: The disease dynamics are temporally periodic and spatially homogeneous. This corresponds to a situation where the entire region experiences synchronous epidemic cycles, such as recurrent waves of infection due to seasonality or behavioral cycles, but without spatial differentiation.

{\it Post-transition Stage ($t>2800$)}: The system undergoes a transition from temporal-only oscillations to spatiotemporal pattern formation. In practical terms, it means the disease begins to localize in space, forming hotspots or periodic waves of infection across the domain. 

\begin{figure}[H]
    \centering
 \subfigure[Surface of $S(x, t)$]{%
         \includegraphics[width=0.45\linewidth]{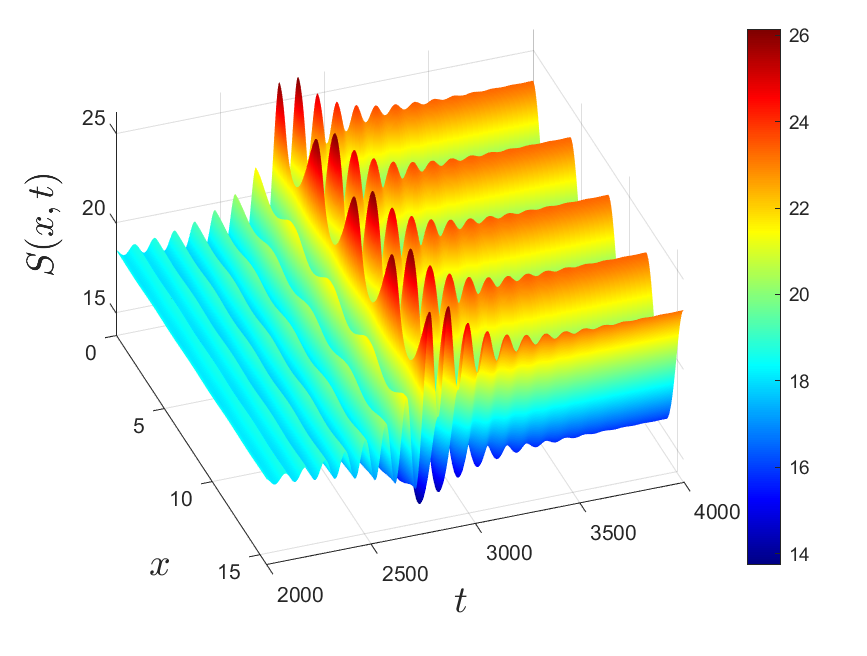}%
    }
    \hfill
    \subfigure[Surface of $I(x, t)$]{
         \includegraphics[width=0.45\linewidth]{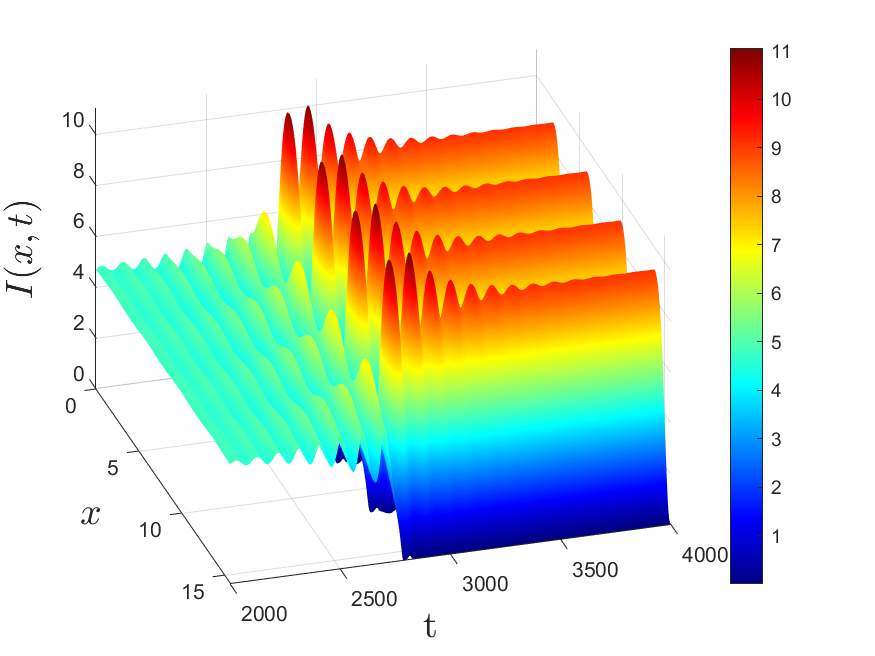}%
      }
    \caption{\footnotesize Transient dynamics from temporal oscillatory regime to a spatially periodic pattern.  The structure grows in amplitude, clearly indicating a transition from a homogeneous oscillatory state to a spatially patterned wave.}
    \label{surface_transient}
\end{figure}

\begin{figure}[H]
    \centering
 \subfigure[Projection of $S(x, t)$]{%
         \includegraphics[width=0.45\linewidth]{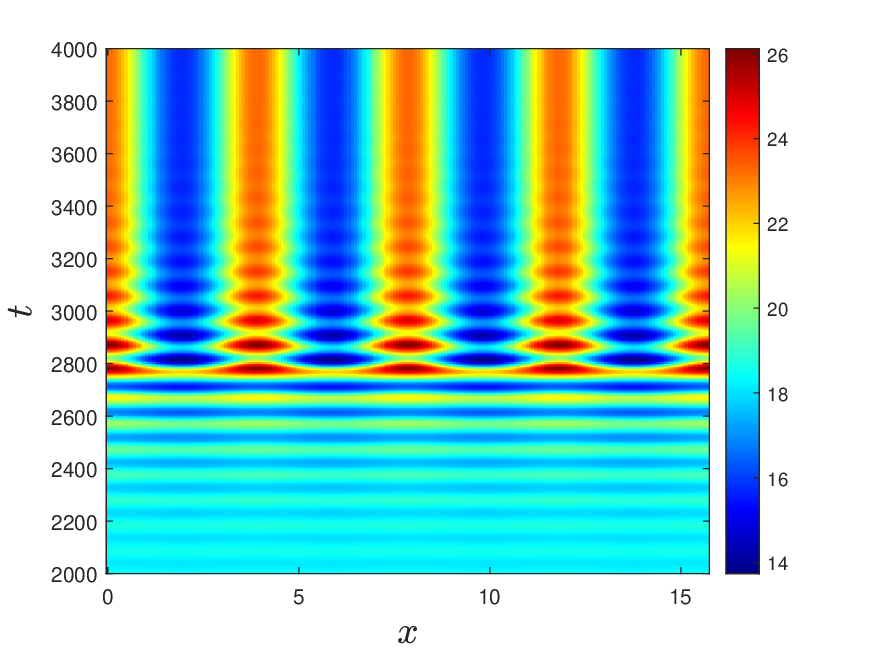}%
    }
    \hfill
    \subfigure[Projection of $I(x, t)$]{
         \includegraphics[width=0.45\linewidth]{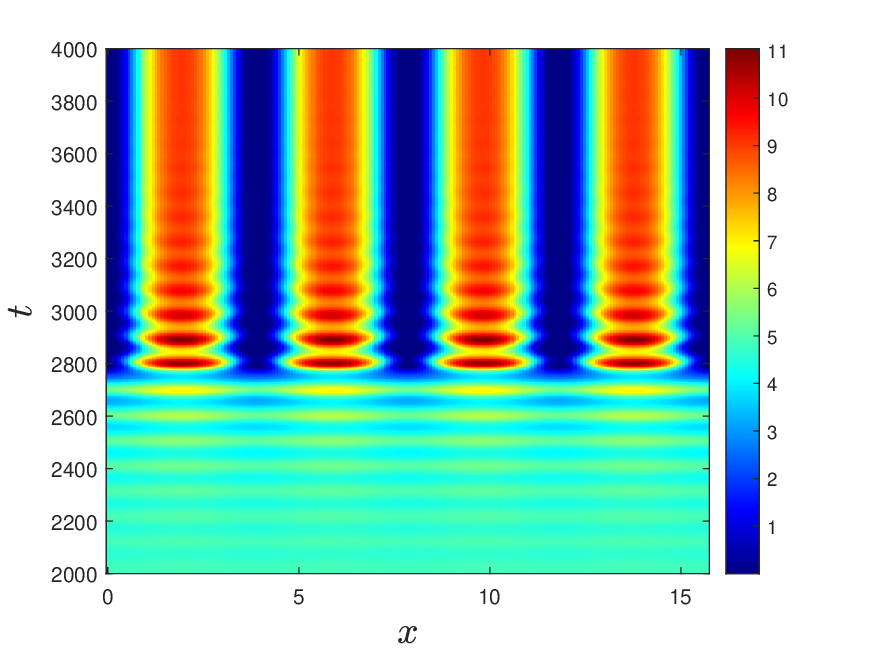}%
      }
    \caption{\footnotesize The heatmap displays a transition from a homogeneous strip pattern at early times to regular vertical stripes at later times, indicating that the solution becomes spatially periodic from temporal oscillation. }
    \label{projection_transient}
\end{figure}

{\bf (d) Public health implication}. The Turing pattern, spatiotemporal pattern, and transient dynamics suggest that uniform interventions (e.g., vaccination or control policies applied evenly across space) may not be effective. Instead, spatially targeted strategies are necessary to contain disease waves that vary regionally and cyclically.

In summary, we have investigated a diffusive epidemic model with an infection-dependent recovery rate and analyzed its pattern formation behavior using bifurcation theory and numerical simulations. The model exhibits rich spatiotemporal dynamics, including multiple steady states, spatially homogeneous periodic oscillations, and the emergence of spatial and spatiotemporal patterns induced by Turing and Turing–Hopf bifurcations. From an epidemiological perspective, these patterns explain localized hotspots of infection and asynchronous disease recurrence across spatial regions. These findings underscore the importance of incorporating spatial heterogeneity and nonlinear recovery mechanisms into epidemic models to better understand real-world disease dynamics. Such insights can inform more effective, spatially targeted intervention strategies in the control of infectious diseases.

\vskip 0.5cm

\section*{Appendix.}

{\it The proof of Lemma \ref{root}.} By straightforward calculation, we have
$$\Psi'(I)=3\beta I^2+2(2b\beta+d)I+(b^2\beta+2bd-b\mu_1+b\mu_0).$$

(1). The discriminant of $\Psi'(I)$ is $4(b\beta-d)^2+12b\beta(\mu_1-\mu_0)$, which is positive.

(2). If $\bar{\omega}<b\beta+2d$, then $\Psi'(0)>0$. Hence, $\Psi'(I)$ has two negative roots. Since $\Psi(0)>0$, $\Psi(I)$ does not have positive roots, so $\Psi(I)>0$ for all $I>0$.

(3). If $\bar{\omega}=b\beta+2d$, then $\Psi'(0)=0$. Hence, $\Psi'(I)$ has a zero root and a negative root. Since $\Psi(0)>0$, $\Psi(I)$ does not have positive roots, so $\Psi(I)>0$ for all $I>0$.

(4). If $\bar{\omega}>b\beta+2d$, then $\Psi'(0)<0$. Hence, $\Psi'(I)$ has a positive root and a negative root. The discriminant of $\Psi(I)$ is
$$\tilde{\Delta}(\bar{\omega})=4b\beta\bar{\omega}^2+(d^2-20db\beta-8b^2\beta^2)\bar{\omega}+4(b\beta-d)^3.$$
Regard $\tilde{\Delta}(\bar{\omega})$ as a function of $\bar{\omega}$. Note that $\tilde{\Delta}=-d(2d^2+11db\beta+32b^2\beta^2)<0$ if $\bar{\omega}=b\beta+2d$, so $\tilde{\Delta}(\bar{\omega})$ has two roots, where one root is less than $b\beta+2d$ and the other root is greater than $b\beta+2d$, which is $\omega$.

If $\bar{\omega}\in(b\beta+2d, \omega)$, then $\tilde{\Delta}<0$,  so $\Psi(I)$ does not have positive roots, so $\Psi(I)>0$ for all $I>0$. If $\bar{\omega}=\omega$, then $\tilde{\Delta}=0$,  so $\Psi(I)$ has a double root $\tilde{I}>0$, and $\Psi(I)>0$ for $I>0$ and $I\in(0, \tilde{I})\cup(\tilde{I}, \infty)$. If $\bar{\omega}\in(\omega, \infty)$, then $\tilde{\Delta}>0$,  so $\Psi(I)$ has distinct positive roots $\check{I}$ and $\hat{I}$, and $\Psi(I)>0$ for $I\in (0, \check{I})\cup(\hat{I}, \infty)$ and $\Psi(I)<0$ for $I\in(\check{I}, \hat{I})$.




\vskip 0.3cm

\noindent {\bf Data Availability} No datasets were generated or analyzed during the current study.






\end{document}